\documentclass[9pt,a4paper]{article}
\usepackage{amssymb}
\usepackage{mathrsfs}
\usepackage{amsfonts}
\usepackage{amsthm,amscd,amsmath}
\usepackage{eufrak}

\def\Ext{\textrm{Ext}}
\def\Hom{\textrm{Hom}}
\def\End{\textrm{End}}
\def\Aut{\textrm{Aut}}
\def\dim{\textrm{dim}}
\def\rad{\textrm{rad}}
\def\dimv{\underline{\rm dim}}

\author{Yong Jiang \and Jie Sheng \and Jie Xiao}
\title{The elements in crystal bases corresponding to exceptional modules\footnote{ The research was supported in part by NSF of China (No.
10631010) and by NKBRPC (No. 2006CB805905)}}

\newtheorem{thm}{Theorem}[section]
\newtheorem{prop}[thm]{Proposition}
\newtheorem{lem}[thm]{Lemma}

\newtheorem{rem}[thm]{Remark}

\begin{document}

\maketitle

\begin{abstract}
According to the Ringel-Green Theorem(\cite{g},\cite{r1}), the
generic composition algebra of the Hall algebra provides a
realization of the positive part of the quantum group. Furthermore,
its Drinfeld double can be identified with the whole quantum
group(\cite{x},\cite{xy}), in which the BGP-reflection functors
coincide with Lusztig's symmetries. We first assert the elements
corresponding to exceptional modules lie in the integral generic
composition algebra, hence in the integral form of the quantum
group. Then we prove that these elements lie in the crystal basis up
to a sign. Eventually we show that the sign can be removed by the
geometric method. Our results hold for any type of Cartan datum.
\end{abstract}

\section{Introduction}\label{section 1}

Let $\Delta$ be a symmetrizable generalized Cartan matrix, or
$\Delta=(I, (-,-))$ a Cartan datum in the sense of Lusztig
\cite{l1}, $\mathfrak{g}$ the corresponding symmetrizable Kac-Moody
algebra. We have the Drinfeld-Jimbo quantized enveloping algebra, or
the quantum group, $U=U_{q}(\mathfrak{g})$ attached to the Cartan
datum $\Delta$. Lusztig gave it a series of important automorphisms,
now called Lusztig's symmetries. By applying Lusztig's symmetries
and the induced action of the braid group on $U^+,$ Lusztig found an
algebraic approach to construct the canonical basis of $U^+$ in
finite type. If $\Delta$ is of infinite type, there are much more
root vectors beyond the set obtained by applying the braid group
action on Chevalley generators of $U^+.$ However Lusztig's geometric
method by using perverse sheaves and intersection cohomology to
construct the canonical bases works for general infinite type.  A
different algebraic approach due to Kashiwara works for arbitrary
type. He constructed the crystal basis and the global crystal basis
($=$ canonical basis) of the negative part $U^{-}$ of the quantum
group. Roughly speaking, the crystal basis is a good basis of the
quantum group at $q=0.$

Given a Cartan datum $\Delta$, there is a finite dimensional
hereditary $k$-algebra $\Lambda$ to realize it, where $k$ is a
finite field. Then we have the corresponding Hall algebra
$\mathscr{H}(\Lambda)$. According to the Ringel-Green Theorem(see
\cite{g},\cite{r1}), the generic composition algebra
$\mathscr{C}(\Delta)$ of $\mathscr{H}(\Lambda)$, precisely of the
Cartan datum $\Delta$, provides a realization of the positive part
$U^{+}$ of the quantum group corresponding to $\Delta$. With the
comultiplication of $\mathscr{H}(\Lambda)$ given by Green in
\cite{g}, it is natural to obtain a Hopf algebra structure of
$\mathscr{H}(\Lambda)$ by adding a torus, and then to consider the
Drinfeld double of the Hall algebra (see \cite{x}). Therefore, the
Drinfeld double of the generic composition algebra
$\mathcal{D}_{\mathscr{C}}(\Delta)$ provides a realization of the
whole $U$. This realization builds up a bridge between the quantum
groups and the representation theory of hereditary algebras.
Especially, it connects Lusztig's symmetries with BGP-reflection
functors (see \cite{xy}), so this two important operators can be
considered simultaneously.

In this article, we consider a special family of elements $\{\langle
u_{\lambda}\rangle |\lambda \in \mathcal{E}\}$ in the Hall algebra,
where $\mathcal{E}$ is the set of isomorphism classes of all
exceptioal $\Lambda$-modules. These elements are much more than the
elements obtained by applying the braid group action on Chevalley
generators of $U^+$. From the work of \cite{z} and \cite{cx} we know
that these elements lie in the generic composition algebra
$\mathcal{D}_{\mathscr{C}}(\Delta)$. Our first result asserts in
Theorem \ref{thm 6 exc in integral comp alg} that each $\langle
u_{\lambda}\rangle $ lies in the integral generic composition
algebra $\mathscr{C}_{\mathbb{Z}}(\Delta)$, hence in the integral
form of the positive part of the quantum group $U^{+}_{\mathbb{Z}}$
(by identifying $\mathcal{D}_{\mathscr{C}}(\Delta)$ with $U$).

The main goal of us is to relate the exceptional modules with
Kashiwara's crystal bases. For convenience, we use the the crystal
structure $(L(\infty),B(\infty))$ in $U^{+}$ instead of in $U^{-}$.
Our main result is (see Theorem \ref{thm 6 exc in crystal basis})
that the image of $\langle u_{\lambda}\rangle$ in
$L(\infty)/qL(\infty)$ lies actually in the crystal basis
$B(\infty)$ up to a sign. In the last section we remove the sign by
comparing with Lusztig's geometric method. Therefore the image of
$\langle u_{\lambda}\rangle$ in $L(\infty)/qL(\infty)$ always
belongs to $B(\infty)$.

The organization of this paper is as follows: In Section
\ref{section 2}, we review some basic facts of quantum groups and
crystal bases. Then the polarization, which will be called
Kashiwara's pairing, is defined in the positive part of quantum
groups. For the details, see \cite{l1} and \cite{k}. In Section
\ref{section 3}, we first give the definitions of Hall algebras and
composition algebras. Following \cite{x}, we concisely restate the
Drinfeld double structure of Hall algebras. Also, we get some
important operators $r'_{i}$ in the Hall algebras. Then $r'_{i}$ is
the same as the derivation operators $f'_{i}$ when we identify the
generic composition algebra with $U^{+}$. The aim of Section
\ref{section 4} is to obtain $\langle u_{\lambda}\rangle$
corresponding to preprojective or preinjective modules from the
simple modules, after establishing the isomorphism between Lusztig's
symmetries and reflection functors. The results come from
\cite{bgp}, \cite{r5} and \cite{xy}. Section \ref{section 5} gives a
review of an algorithm in \cite{cx}. This algorithm comes from a
result of Crawley-Boevey \cite{cb}, to state that any exceptional
module can be obtained inductively from simple modules using braid
group actions on exceptional sequences. The main results will be
stated in Section \ref{section 6}. In Section \ref{section 7} we
prove the first main result by a combination of algorithms in
Section \ref{section 4} and \ref{section 5}. Then in Section
\ref{section 8}, we introduce Ringel's pairing in $U^{+}$, and
compare it with Kashiwara's pairing. Our second main result follows
from direct calculations of Ringel's pairing. However, we need to
remove the sign, which will be done using geometric methods in the
last section.

\section{Quantum groups and crystal bases}\label{section 2}

\textbf{2.1 Quantum groups.} \ \ Let $(I,(-,-))$ be a \emph{Cartan
datum} in the sense of Lusztig. i.e. $I$ is a finite set and $(-,-)$
is a symmetric bilinear form
$\mathbb{Z}[I]\times\mathbb{Z}[I]\rightarrow\mathbb{Z}$ which
satisfies the following conditions:

(a) $(i,i)\in\{2,4,6,...\}$ for any $i\in I$.

(b) $2(i,j)/(i,i)\in \{0,-1,-2,...\}$ for any $i\neq j$ in $I$.

Let $Q=\mathbb{Z}[I]$, $Q_{+}=\mathbb{N}[I]$. $Q$ is called the
\emph{root lattice}. For any $i\in I$ define $s_{i}:Q\rightarrow Q$
by $s_{i}(\mu)=\mu-\dfrac{2(\mu,i)}{(i,i)}i$. $s_{i}$ is called a
\emph{simple reflection}. The \emph{Weyl group} $W$ is defined to be
the group generated by all the simple reflections.

Note that we can identify a Cartan datum $(I,(-,-))$ with a
\emph{symmetrizable generalized Cartan matrix}
$\Delta=(a_{ij})_{i,j}$ by setting $a_{ij}=2(i,j)/(i,i)$. Let
$\varepsilon_{i}=(i,i)/2$, then $(\varepsilon_{i})_{i}$ is the
minimal symmetrization. We have the corresponding Kac-Moody algebra
$\mathfrak{g}$.

Let $\mathbb{Q}(v)$ be the function field in one variable $v$ over
$\mathbb{Q}$. The Drinfeld-Jimbo quantum group
$U=U_{q}(\mathfrak{g})$ is defined to be the associative
$\mathbb{Q}(v)$-algebra with generators $E_{i},F_{i},K_{\mu}$,
($i\in I$ and $\mu\in Q$) subject to the relations:
\begin{displaymath}
K_{\nu}K_{\mu}=K_{\mu}K_{\nu}=K_{\mu+\nu}, \ \ \ K_{0}=1,
\end{displaymath}
\begin{displaymath}
K_{\mu}E_{j}=v^{(\mu,j)}E_{j}K_{\mu}, \ \ \
K_{\mu}F_{j}=v^{-(\mu,j)}F_{j}K_{\mu}
\end{displaymath}
\begin{displaymath}
E_{i}F_{j}-F_{j}E_{i}=\delta_{ij}\dfrac{K_{i}-K_{i}^{-1}}{v_{i}-v_{i}^{-1}},
\end{displaymath}
\begin{displaymath}
\sum_{t=0}^{1-a_{ij}}(-1)^{t}\begin{bmatrix} 1-a_{ij}
\\ t \end{bmatrix}_{\varepsilon_{i}} E_{i}^{t}E_{j}E_{i}^{1-a_{ij}-t}=0, (i\neq
j)
\end{displaymath}
\begin{displaymath}
\sum_{t=0}^{1-a_{ij}}(-1)^{t}\begin{bmatrix} 1-a_{ij}
\\ t \end{bmatrix}_{\varepsilon_{i}} F_{i}^{t}F_{j}F_{i}^{1-a_{ij}-t}=0,
(i\neq j).
\end{displaymath}
where $v_{i}=v^{\varepsilon_{i}}$ and we use the notations
\begin{displaymath}
[n]=\dfrac{v^{n}-v^{-n}}{v-v^{-1}}=v^{n-1}+v^{n-3}+\cdots+v^{-n+1},
\end{displaymath}
\begin{displaymath}
[n]!=\prod_{r=1}^{n}[r],\ \ \
\begin{bmatrix} n \\ r \end{bmatrix}=\dfrac{[n]!}{[r]![n-r]!}.
\end{displaymath}
and for any polynomial $f \in \mathbb{Z}[v,v^{-1}]$ and an integer
$a$, we denote by $f_{a}$ the polynomial obtained from $f$ by
replacing $v$ by $v^{a}$.

The following elementary properties of $U$ are well known:

(a) $U$ has a triangular decomposition
\begin{displaymath}
U\cong U^{-}\otimes U^{0}\otimes U^{+},
\end{displaymath}
where $U^{+}$ (resp. $U^{-}$) is the subalgebra of $U$ generated by
$E_{i}$ (resp. $F_{i}$), $i\in I$, and $U^{0}$ is the subalgebra
generated by $K_{i}^{\pm}$,$i\in I$.

(b) $U^{+}$ and $U^{-}$ are $Q_{+}$-graded algebras, i.e.
\begin{displaymath}
U^{+}=\bigoplus_{\nu\in Q_{+}}U^{+}_{\nu}, \ \
U^{-}=\bigoplus_{\nu\in Q_{+}}U^{-}_{-\nu}.
\end{displaymath}
where $U^{\pm}_{\nu}=\{u\in
U^{\pm}|K_{i}uK_{i}^{-1}=v^{\pm(i,\nu)}u, \text{for any}\ \ i\in
I\}$.

(c) $U$ has a Hopf algebra structure (See \cite{l1}).

Lusztig introduced an important family of automorphisms
$T_{i,1}'':U\rightarrow U$ called the \emph{symmetries}. They are
defined by
\begin{displaymath}
T_{i,1}''(E_{i})=-F_{i}K_{i}^{\varepsilon_{i}},
\end{displaymath}
\begin{displaymath}
T_{i,1}''(F_{i})=-K_{i}^{-\varepsilon_{i}}E_{i},
\end{displaymath}
\begin{displaymath}
T_{i,1}''(E_{j})=\sum_{r+s=-a_{ij}}(-1)^{r}v^{-r\varepsilon_{i}}E_{i}^{(s)}E_{j}E_{i}^{(r)},\text{
for } i\neq j
\end{displaymath}
\begin{displaymath}
T_{i,1}''(F_{j})=\sum_{r+s=-a_{ij}}(-1)^{r}v^{r\varepsilon_{i}}F_{i}^{(s)}F_{j}F_{i}^{(r)},\text{
for } i\neq j
\end{displaymath}
\begin{displaymath}
T_{i,1}''(K_{\mu})=K_{s_{i}\mu}.
\end{displaymath}

The inverse of $T_{i,1}''$ is $T_{i,-1}'$ (See \cite{l1}).\\

\textbf{2.2 Crystal bases of $U^{+}$.} \ \  We will briefly recall
the definition of crystal bases following Kashiwara \cite{k}.
However, for later convenience, we will state the results in $U^{+}$
rather than $U^{-}$.

There are two operators $f_{i}'$ and $f_{i}''$ on $U^{+}$. They can
be defined inductively as following:
\begin{displaymath}
f_{i}'(1)=f_{i}''(1)=0
\end{displaymath}
\begin{displaymath}
f_{i}'(E_{j})=\delta_{ij}, \ \
f_{i}'(E_{j}P)=v_{i}^{a_{ij}}E_{j}f_{i}'(P)+\delta_{ij}P,
\end{displaymath}
\begin{displaymath}
f_{i}''(E_{j})=\delta_{ij}, \ \
f_{i}''(E_{j}P)=v_{i}^{-a_{ij}}E_{j}f_{i}''(P)+\delta_{ij}P.
\end{displaymath}

According to \cite{k} we have $U^{+}=\bigoplus_{n\geq
0}E_{i}^{(n)}\ker f_{i}'$. Hence we can define the
$\mathbb{Q}(v)$-linear maps $\tilde{E_{i}}$ and $\tilde{F_{i}}$ of
$U^{+}$ by
\begin{displaymath}
\tilde{E_{i}}(E_{i}^{(n)}u)=E_{i}^{(n+1)}u,
\end{displaymath}
\begin{displaymath}
\tilde{F_{i}}(E_{i}^{(n)}u)=E_{i}^{(n-1)}u.
\end{displaymath}
for any $u\in \ker f_{i}'$.

$\tilde{F_{i}}$ and $\tilde{E_{i}}$ are called \emph{Kashiwara's
operators}.

Let $A=\mathbb{Q}[[v^{-1}]]\cap\mathbb{Q}(v)$. A pair $(L,B)$ is
called a \emph{crystal basis} of $U^{+}$ if it satisfies the
following conditions:

(1) $L$ is a free $A$-submodule of $U^{+}$ such that $U^{+}\cong
\mathbb{Q}(v)\otimes_{A}L$.

(2) $B$ is a basis of the $\mathbb{Q}$-vector space $L/v^{-1}L$.

(3) Let $L_{\nu}=L\cap U^{+}_{\nu}$ and $B_{\nu}=B\cap
(L_{\nu}/v^{-1}L_{\nu})$, we have $L=\bigoplus_{\nu\in Q_{+}}
L_{\nu}$, $B=\bigsqcup_{\nu\in Q_{+}} B_{\nu}$.

(4) $\tilde{E_{i}}L\subset L$ and $\tilde{F_{i}}L\subset L$ for any
$i$. (Therefore $\tilde{E_{i}}$ and $\tilde{F_{i}}$ act on
$L/v^{-1}L$)

(5) $\tilde{F_{i}}B\subset B\cup \{0\}$ and $\tilde{E_{i}}B\subset
B$.

(6) For any $b\in B$ such that $\tilde{F_{i}}b \in B$, we have
$\tilde{E_{i}}\tilde{F_{i}}b=b$.

The following theorem asserts the existence of the crystal basis in
$U^{+}$.

\begin{thm}\label{thm 2 exist crystal basis}
Let $L(\infty)$ be the $A$-submodule of $U^{+}$ generated by
$\tilde{E}_{i_{1}}\tilde{E}_{i_{2}}\cdots \tilde{E}_{i_{l}}\cdot 1$
and $B(\infty)$ be the subset of $L(\infty)/v^{-1}L(\infty)$
consisting of the nonzero vectors of the form
$\tilde{E}_{i_{1}}\tilde{E}_{i_{2}}\cdots \tilde{E}_{i_{l}}\cdot
\bar{1}$.

Then $(L(\infty),B(\infty))$ is the crystal basis of $U^{+}$.
\end{thm}

\textbf{2.3 A characterization of $B(\infty)$.} \ \ Kashiwara has
given a nice characterization of the crystal basis using the
$\mathbb{Z}$-form and a paring $(-,-)_{K}$.

\begin{prop}\label{prop 2 Kashiwara form}
(a) There is a unique non-degenerate symmetric
$\mathbb{Q}(v)$-bilinear pairing $(-,-)_{K}$ on $U^{+}$ such that
\begin{displaymath}
(1,1)_{K}=1,
\end{displaymath}
\begin{displaymath}
(E_{i}x,y)_{K}=(x,f_{i}'(y))_{K}.
\end{displaymath}
(b) We have $(L(\infty),L(\infty))_{K}\subset A$.
\end{prop}

Part (b) of this proposition implies that the form $(-,-)_{K}$
induces a $\mathbb{Q}$-bilinear form $(-,-)_{K,0}$ on
$L(\infty)/v^{-1}L(\infty)$:
\begin{displaymath}
(x+v^{-1}L(\infty),y+v^{-1}L(\infty))_{K,0}=(x,y)_{K}+v^{-1}A
\end{displaymath}
for any $x,y\in L(\infty)$.

\begin{prop}\label{prop 2 char of L}
(a) For any $b_{1},b_{2}\in B(\infty)$,
$(b_{1},b_{2})=\delta_{b_{1}b_{2}}$, i.e. $B(\infty)$ is an
orthogonal normal basis of $L(\infty)/v^{-1}L(\infty)$ with respect
to the form $(-,-)_{K,0}$. In particular, $(-,-)_{K,0}$ is positive
definite.

(b) $L(\infty)=\{u\in U^{+}|(u,u)_{K}\in A\}$.
\end{prop}

Set $E_{i}^{(n)}=E_{i}^{n}/[n]_{\epsilon_{i}}!$,
$F_{i}^{(n)}=F_{i}^{n}/[n]_{\epsilon_{i}}!$. Denote by
$U_{\mathbb{Z}}$ the $\mathbb{Z}[v,v^{-1}]$-subalgebra of $U$
generated by $F_{i}^{(n)}$, $E_{i}^{(n)}$ and $K_{\mu}$, $(i\in I, \
\mu \in Q)$. Let $U_{\mathbb{Z}}^{+}$ (resp. $U_{\mathbb{Z}}^{-}$)
be the $\mathbb{Z}[v,v^{-1}]$-subalgebra of $U^{+}$ (resp. $U^{-}$)
generated by $E_{i}^{(n)}$ (resp. $F_{i}^{(n)}$). Then it is easy to
see that
\begin{displaymath}
U_{\mathbb{Z}}^{-}=U_{\mathbb{Z}}\cap U^{-},\ \ \
U_{\mathbb{Z}}^{+}=U_{\mathbb{Z}}\cap U^{+}.
\end{displaymath}

Lusztig's symmetries also act nicely on $U_{\mathbb{Z}}$, so
actually $T_{i,1}''$ and $T_{i,-1}'$ are automorphisms on
$U_{\mathbb{Z}}$. (See \cite{l1}, 37.1.3)

We have $U_{\mathbb{Z}}^{+}$ is stable under $f_{i}'$ and
Kashiwara's operators $\tilde{E_{i}}$, $\tilde{F_{i}}$.

Set $L_{\mathbb{Z}}(\infty)=L(\infty)\cap U_{\mathbb{Z}}^{+}$, then
$L_{\mathbb{Z}}(\infty)$ is stable under $\tilde{E_{i}}$ and
$\tilde{F_{i}}$.

\begin{prop}\label{prop 2 char of B}
(a) $(-,-)_{K,0}$ is $\mathbb{Z}$-valued on
$L_{\mathbb{Z}}(\infty)/v^{-1}L_{\mathbb{Z}}(\infty)$.

(b) $L_{\mathbb{Z}}(\infty)/v^{-1}L_{\mathbb{Z}}(\infty)$ is a free
$\mathbb{Z}$-module with $B(\infty)$ as a basis.

(c) $B(\infty)\cup(-B(\infty))=\{u\in
L_{\mathbb{Z}}(\infty)/v^{-1}L_{\mathbb{Z}}(\infty)|(u,u)_{K,0}=1\}$.
\end{prop}

\section{Hall algebras and Drinfeld double}\label{section 3}

\textbf{3.1 The Hall algebra of a hereditary algebra.} \ \ Let
$\Lambda$ be a finite-dimensional hereditary $k$-algebra where $k$
is a finite field of $q$ elements. Denote the set of isomorphism
classes of finite-dimensional $\Lambda$-modules by $\mathcal{P}$. We
can choose a representative $V_{\alpha}\in\alpha$ for each
$\alpha\in \mathcal{P}$. Given any $\Lambda$-modules $M$ and $N$, we
have the Euler form:
\begin{displaymath}
\langle M,N \rangle =
\dim_{k}\Hom_{\Lambda}(M,N)-\dim_{k}\Ext_{\Lambda}(M,N).
\end{displaymath}
$\langle M,N \rangle$ depends only on the dimension vectors $ \dimv
M$ and $\dimv N$ since $\Lambda$ is hereditary, so we write $\langle
\alpha,\beta \rangle =\langle V_{\alpha},V_{\beta} \rangle$. The
Euler symmetric form $(-,-)$ is given by $(\alpha,\beta)=\langle
\alpha,\beta\rangle+\langle\beta,\alpha\rangle$. So the Euler form
and the Euler symmetric form are both defined on $\mathbb{Z}[I]$
where $I$ is the set of isomorphism classes of simple
$\Lambda$-modules. Then $\Delta=(I,(-,-))$ is a Cartan datum and any
Cartan datum can be realized by the Euler symmetric form of a
finite-dimensional hereditary $k$-algebra (See \cite{r4}).

For $\alpha,\beta,\lambda \in \mathcal{P}$, let
$g^{\lambda}_{\alpha\beta}$ be the number of submodules $B$ of
$V_{\lambda}$ such that $B\cong V_{\beta}$ and $V_{\lambda}/B \cong
V_{\alpha}$.

Let $v=\sqrt{q}$, the \emph{Hall algebra} of $\Lambda$ is a free
$\mathbb{Q}(v)$-module whose basis consists of isomorphism classes
of $\Lambda$-modules with multiplication defined as
$$u_{\alpha}u_{\beta}=v^{\langle \alpha,\beta
\rangle}\sum_{\lambda\in\mathcal{P}}g_{\alpha\beta}^{\lambda}u_{\lambda},\
\text{for all}\ \alpha,\beta\in \mathcal{P}.$$

We know $\mathscr{H}(\Lambda)$ is an associative
$\mathbb{N}[I]$-graded $\mathbb{Q}(v)$-algebra with the identity
element $u_{0}$ and the grading
$\mathscr{H}(\Lambda)=\bigoplus_{r\in\mathbb{N}[I]}\mathscr{H}_{r}$
where $\mathscr{H}_{r}$ is the $\mathbb{Q}(v)$-span of the set
$\{u_{\lambda}|\lambda\in \mathcal{P},\dimv V_{\lambda}=r\}$. The
$\mathbb{Q}(v)$-subalgebra $\mathscr{C}(\Lambda)$ generated by
$u_{i},i\in I$ is called the \emph{composition algebra} of
$\Lambda$.

In this paper we are dealing with exceptional modules. A
$\Lambda$-module $V_{\lambda} \ (\lambda\in\mathcal{P})$ is called
\emph{exceptional} if $\Ext_{\Lambda}(V_{\lambda},V_{\lambda})=0$,
i.e. $V_{\lambda}$ has no self-extension. For any exceptional module
$V_{\lambda}$, we set
$u^{(t)}_{\lambda}=(1/[t]!_{\varepsilon(\lambda)})u^{t}_{\lambda}$
in the Hall algebra, where
$\varepsilon(\lambda)=\dim_{k}\End_{\Lambda}V_{\lambda}$. We have
the following identities:
$u^{(t)}_{\lambda}=(v^{\varepsilon(\lambda)})^{t(t-1)}u_{t\lambda}$,
where $u_{t\lambda}$ corresponds to the direct sum of $t$ copies of
$V_{\lambda}$.

Now fix a Cartan datum $\Delta$. Let $\overline{k}$ be the algebraic
closure of $k$ and for any $n\in \mathbb{N}$, $F(n)$ be a subfield
of $\overline{k}$ such that $[F(n):k]=n$ .Then we have a
finite-dimensional hereditary $F(n)$-algebra $\Lambda(n)$
corresponding to $\Delta$. Thus we have a series of Hall algebras
$\mathscr{H}=\mathscr{H}(\Lambda(n))$. Define a new ring
$\Pi=\prod_{n>0}\mathscr{H}_{n}$, then $v=(v_{n})_{n}\in\Pi$ where
$v_{n}=\sqrt{|F(n)|}=\sqrt{q^{n}}$. Obviously $v$ is in the center
of $\Pi$ and transcendental over $\mathbb{Q}$. Denote
$u_{i}=(u_{i}(n))_{n}\in\Pi$ where $u_{i}(n)$ is the element of
$\mathscr{H}(\Lambda(n))$ corresponding to the simple
$\Lambda(n)$-module which lies in the class $i\in I$. The subring of
$\Pi$ generated by the elements $v$, $v^{-1}$ and $u_{i}(i\in I)$,
hence a $\mathbb{Q}(v)$-algebra, is called the \emph{generic
composition algebra} of the Cartan datum $\Delta$. We will denote it
by $\mathscr{C}(\Delta)$.

On the other hand, we have $U^{+}$, the positive part of the
Drinfeld-Jimbo quantum group corresponding to the Cartan datum
$\Delta$. By Green \cite{g} and Ringel \cite{r1} we know that
$\mathscr{C}(\Delta)$ is isomorphic to $U^{+}$ as associative
$\mathbb{Q}(v)$-algebras, where $u_{i}$ is sent to $E_{i}$ for each
$i\in I$. Therefore, corresponding to the $\mathbb{Z}$-form of
quantum groups, we can define $\mathscr{C}_{\mathbb{Z}}(\Delta)$ to
be the $\mathbb{Z}[v,v^{-1}]$-subalgebra of $\mathscr{C}(\Delta)$
generated by $u^{(t)}_{i}, i\in I, t\in\mathbb{N}$, which will be
called the \emph{integral generic composition algebra} of the Cartan
datum $\Delta$. Obviously $\mathscr{C}_{\mathbb{Z}}(\Delta)$ is
isomorphic to $U^{+}_{\mathbb{Z}}$.\\

\textbf{3.2 The Drinfeld double.} \ \ In the Hall algebra
$\mathscr{H}(\Lambda)$, we write $\langle u_{\alpha} \rangle =
v^{-\dim_{k} V_{\alpha}+\varepsilon(\alpha)}u_{\alpha}$ for each
$\alpha \in \mathcal{P}$. Here $\varepsilon(\alpha)=\langle \alpha,
\alpha \rangle$, which is equal to
$\dim_{k}\End_{\Lambda}V_{\alpha}$ when $V_{\alpha}$ is exceptional.
Then $\mathscr{H}(\Lambda)$ can be viewed as a  free
$\mathbb{Q}(v)$-algebra with basis $\langle u_{\alpha} \rangle,
\alpha \in \mathcal{P}$. The multiplication formula can be replaced
by
\begin{displaymath}
\langle u_{\alpha}\rangle \langle u_{\beta}\rangle = v^{-\langle
\beta,\alpha\rangle}\sum_{\lambda\in\mathcal{P}}g^{\lambda}_{\alpha\beta}\langle
u_{\lambda}\rangle \ \ \ \textmd{for all } \alpha, \beta \in
\mathcal{P}
\end{displaymath}

Now we introduce the extended Hall algebra $\mathcal{H}(\Lambda)$ by
adding a torus to $\mathscr{H}(\Lambda)$. Let $\mathcal{H}(\Lambda)$
be the free $\mathbb{Q}(v)$-module with the basis
$$\{K_{\alpha}\langle u_{\lambda}\rangle| \alpha \in \mathbb{Z}[I],\lambda\in \mathcal{P}\}.$$
and extend the multiplication by
\begin{gather*}
K_{\alpha}\langle u_{\beta} \rangle = v^{(\alpha,\beta)}\langle u_{\beta}\rangle K_{\alpha}\ \ \  \textmd{for all } \alpha \in \mathbb{Z}[I], \beta \in \mathcal{P}\\
K_{\alpha}K_{\beta}=K_{\alpha+\beta} \ \ \ \textmd{for all }
\alpha,\beta \in \mathbb{Z}[I]
\end{gather*}
Moreover, $\mathcal{H}(\Lambda)$ has been equipped with a Hopf
algebra structure by Green's comultiplication and an antipode (See
\cite{g}, \cite{x}).

Let $\mathcal{H}^{+}(\Lambda)$ be the Hopf algebra
$\mathcal{H}(\Lambda)$ above but we write $\langle
u^{+}_{\lambda}\rangle $ for $\langle u_{\lambda}\rangle $ for all
$\lambda\in \mathcal{P}$. Dually, we can define
$\mathcal{H}^{-}(\Lambda)$ to be the free $\mathbb{Q}(v)$-module
with the basis $ \{K_{\alpha}\langle u^{-}_{\lambda} \rangle |\alpha
\in \mathbb{Z}[I],\lambda \in \mathcal{P}\}$.
$\mathcal{H}^{-}(\Lambda)$ has a similar Hopf algebra structure (See
\cite{x} or \cite{xy}).

In view of \cite{x}, we obtain the \emph{the reduced Drinfeld
double} $\mathcal{D}(\Lambda)$ coming from a Hopf algebra structure
of $\mathcal{H}^{+}(\Lambda)\otimes \mathcal{H}^{-}(\Lambda)$, by
means of a skew Hopf paring on $\mathcal{H}^{+}(\Lambda)\times
\mathcal{H}^{-}(\Lambda)$. Then there exists the reduced Drinfeld
double $\mathcal{D}_{\mathscr{C}}(\Delta)$ of the generic
composition algebra which is generated by $u^{\pm}_{i},i \in I$, and
$K_{\alpha}, \alpha\in \mathbb{Z}[I]$. Then
$\mathcal{D}_{\mathscr{C}}(\Delta)$ has the triangular decomposition
$\mathcal{D}_{\mathscr{C}}(\Delta)=\mathscr{C}^{-}(\Delta)\otimes T
\otimes \mathscr{C}^{+}(\Delta)$, where $\mathscr{C}^{-}(\Delta)$ is
the subalgebra generated by $u^{-}_{i},i \in I$,
$\mathscr{C}^{+}(\Delta)$ the subalgebra generated by $u^{+}_{i},i
\in I$, and $T$ the torus algebra.

\begin{thm}\label{thm 3 Drinfeld d isom quantum group}
(See \cite{x}) \ The map $\theta :
\mathcal{D}_{\mathscr{C}}(\Delta)\rightarrow U$ by sending
\begin{displaymath}
\langle u^{+}_{i} \rangle \rightarrow E_{i}, \langle
u^{-}_{i}\rangle \rightarrow -v^{\varepsilon(i)}F_{i} , K_{i}
\rightarrow K_{i}^{\varepsilon(i)}
\end{displaymath}
for all $i \in I$ induces an isomorphism as Hopf
$\mathbb{Q}(v)$-algebras.
\end{thm}

\textbf{3.3 Some Derivations.} \ \ For $\alpha \in \mathcal{P}$, we
define the following operators on $\mathscr{H}(\Lambda)$:
$$r_{\alpha}(\langle u_{\lambda} \rangle)= \sum_{\beta\in \mathcal{P}}v^{\langle \beta, \alpha \rangle+(\alpha,\beta)}g^{\lambda}_{\beta\alpha}\frac{a_{\beta}a_{\alpha}}{a_{\lambda}}\langle u_{\beta}\rangle$$
$$r'_{\alpha}(\langle u_{\lambda} \rangle)= \sum_{\beta\in \mathcal{P}}v^{\langle \alpha,\beta  \rangle+(\alpha,\beta)}g^{\lambda}_{\alpha\beta}\frac{a_{\beta}a_{\alpha}}{a_{\lambda}}\langle u_{\beta}\rangle$$
for all $\lambda \in \mathcal{P}$, where
$a_{\alpha}=\Aut_{\Lambda}(V_{\alpha})$.

The following lemma can be proved by direct calculation (similar to
\cite{cx}, Prop 3.2).

\begin{lem}\label{lem 3 derivation}
For any $i\in I$ and $\lambda_{1},\lambda_{2}\in \mathcal{P}$, we
have
\begin{displaymath}
r_{i}(\langle u_{\lambda_{1}}\rangle \langle
u_{\lambda_{2}}\rangle)=\langle u_{\lambda_{1}}\rangle r_{i}(\langle
u_{\lambda_{2}}\rangle)+v^{(i,\lambda_{2})}r_{i}(\langle
u_{\lambda_{1}}\rangle)\langle u_{\lambda_{2}}\rangle.
\end{displaymath}
\begin{displaymath}
r'_{i}(\langle u_{\lambda_{1}}\rangle \langle
u_{\lambda_{2}}\rangle)=v^{(i,\lambda_{1})}\langle
u_{\lambda_{1}}\rangle r'_{i}(\langle
u_{\lambda_{2}}\rangle)+r'_{i}(\langle
u_{\lambda_{1}}\rangle)\langle u_{\lambda_{2}}\rangle.
\end{displaymath}
\end{lem}

By the lemma above it is easy to see that
\begin{displaymath}
r'_{i}(1)=r_{i}'(\langle u_{0}\rangle)=0, \ r'_{i}(\langle
u_{j}\rangle)=\delta_{ij},
\end{displaymath}
\begin{displaymath}
r'_{i}(\langle u_{j}\rangle \langle
u_{\lambda}\rangle)=v^{(i,j)}\langle u_{j}\rangle r'_{i}(\langle
u_{\lambda}\rangle)+\delta_{ij}\langle u_{\lambda}\rangle.
\end{displaymath}

So if we restrict $r'_{i}$ to the composition algebra
$\mathscr{C}(\Lambda)$, we will have
\begin{displaymath}
r'_{i}(\langle u_{j}\rangle P)=v^{(i,j)}\langle u_{j}\rangle
r'_{i}(P)+\delta_{ij}P, \ \ for \ any \ P\in \mathscr{C}(\Lambda)
\end{displaymath}

\begin{lem}\label{lem 3 derivation coincide}
When we identify $U^{+}$ with the generic composition algebra, we
have $f_{i}'=r_{i}'$.
\end{lem}
\begin{proof}
Just compare the above formulas with the definition of $f_{i}'$ in
Section 2.2.
\end{proof}

Also, there are two other derivations on $\mathscr{H}(\Lambda)$
which have nice properties. Namely, define
$$_{\alpha}\delta=\dfrac{(v^{2})^{-\dim_{k}V_{\alpha} +
\varepsilon(\alpha)}}{a_{\alpha}} r'_{\alpha},\
\delta_{\alpha}=\dfrac{(v^{2})^{-\dim_{k}V_{\alpha} +
\varepsilon(\alpha)}}{a_{\alpha}} r_{\alpha}.$$

The key properties of $_{\alpha}\delta$ and $\delta_{\alpha}$ are as
following (but note that here our definitions are slightly different
from the original ones):

\begin{prop}\label{prop 3 derivation homo}
We consider the following linear maps:
\begin{eqnarray*}
\phi_{1}: \mathscr{H}(\Lambda) & \longrightarrow & \Hom_{\mathbb{Q}(v)}(\mathscr{H}(\Lambda),\mathscr{H}(\Lambda))\\
\langle u_{\lambda}\rangle & \longrightarrow &  _{\lambda}\delta\\
\phi_{2}: \mathscr{H}(\Lambda) & \longrightarrow & \Hom_{\mathbb{Q}(v)}(\mathscr{H}(\Lambda),\mathscr{H}(\Lambda))\\
\langle u_{\lambda}\rangle & \longrightarrow &  \delta_{\lambda}
\end{eqnarray*}
(1) $\phi_{1}$
 is an anti-homomorphism, i.e $\phi_{1}(\langle u_{\lambda_{1}}\rangle \langle u_{\lambda_{2}}\rangle)=\phi_{1}(\langle u_{\lambda_{2}}\rangle)\phi_{1}(\langle u_{\lambda_{1}}\rangle)$. \\
(2) $\phi_{2}$ is a homomorphism, i.e $\phi_{2}(\langle
u_{\lambda_{1}}\rangle \langle
u_{\lambda_{2}}\rangle)=\phi_{2}(\langle
u_{\lambda_{1}}\rangle)\phi_{2}(\langle u_{\lambda_{2}}\rangle)$.
\end{prop}

\begin{rem}\label{rem 3 derivation}
From the proposition above, we have
$_{\alpha}\delta(\mathscr{C}(\Lambda))\subseteq
 \mathscr{C}(\Lambda)$ and $\delta_{\alpha}(\mathscr{C}(\Lambda))\subseteq \mathscr{C}(\Lambda)$ if $\langle
u_{\alpha}\rangle\in \mathscr{C}(\Lambda)$. Moreover if $\langle
u_{\alpha}\rangle$ is expressed as a combination of monomials of
$\langle u_{i}\rangle , i \in I$, then $_{\alpha}\delta\ (resp. \
\delta_{\alpha})$ can be expressed as the corresponding combination
of monomials of $_{i}\delta\ (resp. \ \delta_{i}) , i \in I$.
\end{rem}

\section{Reflection functors and Lusztig's symmetries}
\label{section 4} Given a Cartan datum $\Delta$ as before, there is
a valued graph $(\Gamma,d)$ corresponding to it (where
$\Gamma=(\Gamma_{0},\Gamma_{1})$, $\Gamma_{0}$ the set of vertices,
$\Gamma_{1}$ the set of edges with $|\Gamma_{0}|=I$). We obtain
$(\Gamma,d,\Omega)$ by prescribing an orientation $\Omega$ to
$(\Gamma,d)$, and always write $\Omega$ for simplicity. Then let
$\mathscr{S}=(F_{i},{_{i}M_{j}})_{i.j\in \Gamma_{0}}$ be a reduced
k-species of  type $\Omega$. Denote by rep-$\mathscr{S}$ the
category of finite dimensional representations of $\mathscr{S}$ over
$k$. We know that the category rep-$\mathscr{S}$ is equivalent to
the module category of finite dimensional modules over a finite
dimensional hereditary $k$-algebra $\Lambda$. This hereditary
$k$-algebra $\Lambda$ is given by the tensor algebra of
$\mathscr{S}$ (See \cite{dr}). Furthermore, any finite dimensional
hereditary $k$-algebra can be obtained in this way.

Let $p$ be a sink or a source of $\Omega$. We define
$\sigma_{p}\mathscr{S}$ to be the $k$-species obtained from
$\mathscr{S}$ by replacing $_{r}M_{s}$ by its $k$-dual for $r=p$ or
$s=p$, then $\sigma_{p}\mathscr{S}$ is a reduced $k$-species of type
$\sigma_{p}\Omega$, where the orientation $\sigma_{p}\Omega$ is
obtained by reversing the direction of arrows along all edges
containing $p$.

We have the Bernstein-Gelfand-Ponomarev \emph{reflection functors}
$\sigma^{\pm}_{p}: \textmd{rep-}\mathscr{S}$ $\rightarrow
\textmd{rep-}\sigma_{p}\mathscr{S}$, see \cite{bgp}, \cite{dr}.

If $i$ is a vertex of $\Gamma$, let rep-$\mathscr{S}\langle
i\rangle$ be the subcategory of rep-$\mathscr{S}$ consisting of all
representations which do not have $V_{i}$ as a direct summand, where
$V_{i}$ is the simple representation corresponding to $i$. If $i$ is
a sink, then $\sigma^{+}_{i}: \textmd{rep-}\mathscr{S}\langle
i\rangle\rightarrow \textmd{rep-}\sigma_{i}\mathscr{S}\langle
i\rangle$ is an equivalence and it is exact and induces isomorphism
on both $\Hom$ and $\Ext$. The assertion for $\sigma^{-}_{i}$ is the
same if $i$ is a source.

Let $\Lambda$ just be the tensor algebra of a $k$-species
$\mathscr{S}$. We can identify mod-$\Lambda$ with rep-$\mathscr{S}$,
therefore, $\mathscr{H}(\Lambda)$ can be viewed as being defined for
rep-$\mathscr{S}$. We also use $\sigma_{i}\Lambda$ to denote the
tensor algebra of $\sigma_{i}\mathscr{S}$ and $\sigma_{i}\Delta$ to
denote the Cartan datum corresponding to the algebra
$\sigma_{i}\Lambda$ (note that in fact $\sigma_{i}\Delta$ and
$\Delta$ denote the same Cartan datum). We define
$\mathscr{H}(\Lambda)\langle i\rangle$ to be the
$\mathbb{Q}(v)$-subspace of $\mathscr{H}(\Lambda)$ generated by
$\langle u_{\alpha} \rangle$ with $V_{\alpha}\in$
rep-$\mathscr{S}\langle i\rangle$. If $i$ is a sink or a source,
since rep-$\mathscr{S}\langle i\rangle$ is closed under extensions,
$\mathscr{H}(\Lambda)\langle i\rangle$ is a subalgebra of
$\mathscr{H}(\Lambda)$. The following result is due to Ringel
\cite{r5}:

\begin{prop}\label{prop 4 ref functor Hall alg}
Let $i$ be a sink. The functor $\sigma^{+}_{i}$ yields a
$\mathbb{Q}(v)$-algebra isomorphism
$T_{i}:\mathscr{H}(\Lambda)\langle i\rangle\rightarrow
\mathscr{H}(\sigma_{i}\Lambda)\langle i\rangle$ with $T_{i}(\langle
u_{\alpha}\rangle)=\langle u_{\sigma^{+}_{i}\alpha}\rangle$ for any
$V_{\alpha} \in$ rep-$\mathscr{S}\langle i\rangle$.
\end{prop}

The isomorphism $T_{i}$ can be extended to the whole reduced
Drinfeld double $\mathcal{D}(\Lambda)$ as follows (See \cite{xy}):

Let $\overline{K_{i}}=v^{-\varepsilon(i)}K_{i}$, $\langle u_{\alpha}
\rangle ^{(t)}=\langle u_{\alpha} \rangle
^{t}/([t]!)_{\varepsilon(\alpha)}$ for $\alpha\in \mathcal{P}$ and
$t \in \mathbb{N}$.

For $\lambda \in \mathcal{P}$, assume that
$V_{\lambda}=V_{\lambda_{0}}\oplus tV_{i}$ and $V_{\lambda_{0}}$
contains no direct summand isomorphic to $V_{i}$. Then
$\Hom_{\Lambda}(V_{\lambda_{0}},V_{i})=0$ and
$\Ext_{\Lambda}(V_{i},V_{\lambda_{0}})=0$ since $i$ is a sink of
$\mathscr{S}$. Thus $\langle u^{+}_{\lambda}\rangle = v^{\langle
\lambda_{0},ti\rangle}\langle u^{+}_{i}\rangle ^{(t)}\langle
u^{+}_{\lambda_{o}}\rangle$ in $\mathscr{H}^{+}(\Lambda)$. We define
$T_{i}: \mathscr{H}^{+}(\Lambda)\rightarrow
\mathcal{D}(\sigma_{i}\Lambda)$ as follows:
$$T_{i}(\langle
u^{+}_{\lambda}\rangle)=\frac{v^{\langle
\lambda_{0},ti\rangle}}{[t]!_{\varepsilon(i)}}(\langle
u^{-}_{i}\rangle \overline{K_{i}})^{t}\langle
u^{+}_{\sigma^{+}_{i}\lambda_{0}}\rangle =v^{\langle
\lambda,ti\rangle}K_{ti}\langle u^{-}_{i}\rangle ^{(t)}\langle
u^{+}_{\sigma^{+}_{i}\lambda_{0}}\rangle.$$ In particular,
$T_{i}(\langle u^{+}_{i}\rangle)= \langle u^{-}_{i}\rangle
\overline{K_{i}}$.

Symmetrically we define a morphism $T_{i}:
\mathscr{H}^{-}(\Lambda)\rightarrow \mathcal{D}(\sigma_{i}\Lambda)$
as follows:
$$T_{i}(\langle
u^{-}_{\lambda}\rangle)=\frac{v^{\langle
\lambda_{0},ti\rangle}}{[t]!_{\varepsilon(i)}}(\overline{K_{-i}}\langle
u^{+}_{i}\rangle )^{t}\langle
u^{-}_{\sigma^{+}_{i}\lambda_{0}}\rangle =v^{\langle
\lambda,ti\rangle}K_{-ti}\langle u^{+}_{i}\rangle ^{(t)}\langle
u^{-}_{\sigma^{+}_{i}\lambda_{0}}\rangle$$ for all $\lambda \in
\mathcal{P}$, where $V_{\lambda}=V_{\lambda_{0}}\oplus tV_{i}$ and
$V_{\lambda_{0}}$ contains no direct summand isomorphic to $V_{i}$.

Also, we extend $T_{i}$ to the torus algebra by setting
$T_{i}(K_{\alpha})=K_{s_{i}(\alpha)}$ for $\alpha \in
\mathbb{Z}[I]$. Set $T_{i}(K_{\alpha}\langle
u^{\pm}_{\lambda}\rangle)=T_{i}(K_{\alpha})T_{i}(\langle
u^{\pm}_{\lambda}\rangle)$.

The following theorem can be found in \cite{xy}.

\begin{thm}\label{thm 4 ref functotr Drinfeld d}
Let $i$ be a sink. The operator $T_{i}$ induces a
$\mathbb{Q}(v)$-algebra isomorphism:
$\mathcal{D}_{\mathscr{C}}(\Delta)\xrightarrow{\sim}
\mathcal{D}_{\mathscr{C}}(\sigma_{i}\Delta)$.
\end{thm}

If $i$ is a source of $\mathscr{S}$, we can define $T'_{i}$ (via the
reflection functors $\sigma_{i}^{-}$) in the Hall algebra and extend
it to $\mathcal{D}(\Lambda)$ similarly, which also induces a
$\mathbb{Q}(v)$-algebra isomorphism:
$\mathcal{D}_{\mathscr{C}}(\Delta)\rightarrow
\mathcal{D}_{\mathscr{C}}(\sigma_{i}\Delta)$.

Recall that we have the isomorphism
$\mathcal{D}_{\mathscr{C}}(\Delta)\xrightarrow{\sim}U$ in Theorem
\ref{thm 3 Drinfeld d isom quantum group}. So we have the canonical
isomorphism
$\mathcal{D}_{\mathscr{C}}(\Delta)\xrightarrow{\sim}\mathcal{D}_{\mathscr{C}}(\sigma_{i}\Delta)$
by mapping $\langle u^{\pm}_{i}\rangle \mapsto \langle
u^{\pm}_{i}\rangle $ and $K_{i}\mapsto K_{i}$ for a sink $i \in I$.
Therefore we can identify
$\mathcal{D}_{\mathscr{C}}(\sigma_{i}\Delta)$ with
$\mathcal{D}_{\mathscr{C}}(\Delta)$ under this canonical
isomorphism. Then $T_{i}$ induces an automorphism
$\mathcal{D}_{\mathscr{C}}(\Delta)\xrightarrow{\sim}\mathcal{D}_{\mathscr{C}}(\Delta)$.
Similarly $T'_{i}$ can be viewed as an automorphism
$\mathcal{D}_{\mathscr{C}}(\Delta)\xrightarrow{\sim}\mathcal{D}_{\mathscr{C}}(\Delta)$
for a source $i \in I$. The following theorem asserts that $T_{i}$
and $T_{i}'$ coincides with Lusztig's symmetries (see Section
\ref{section 2}).

\begin{thm}\label{thm 4 ref functor coincide symmetries}
(1) Let $i$ be s sink. Then the isomorphism $T_{i}:
\mathcal{D}_{\mathscr{C}}(\Delta)\xrightarrow{\sim}\mathcal{D}_{\mathscr{C}}(\Delta)$
coincides with $T''_{i,1}$. Namely, $T_{i}=T''_{i,1}$, if we
identify $\mathcal{D}_{\mathscr{C}}(\Delta)$ with $U$ by Theorem
\ref{thm 3 Drinfeld d isom quantum group}.

(2) Let $i$ be a source. Then the isomorphism $T'_{i}:
\mathcal{D}_{\mathscr{C}}(\Delta)\xrightarrow{\sim}\mathcal{D}_{\mathscr{C}}(\Delta)$
coincides with $T'_{i,-1}$.
\end{thm}

We keep the notations before. Recall that a sequence
$i_{1},\cdots,i_{m}$ is called a \emph{sink sequence} for $\Omega$,
provided  $i_{1}$ is a sink for $\Omega$, and for $1< t\leq m$, the
vertex $i_{t}$ is a sink for the orientation
$\sigma_{i_{t-1}}\cdots\sigma_{i_{1}}\Omega$. The definition of a
\emph{source sequence} is similar. The following proposition comes
from \cite{r5}.

\begin{prop}\label{prop 4 obtain irregular}
If we identify $\mathcal{D}_{\mathscr{C}}(\Delta)$ with $U$, then:

(1)For any preinjective module $V_{\alpha}$, there exists a source
sequence $i_{1},\cdots,i_{m}$ for $\Omega$ such that $$\langle
u_{\alpha}\rangle= \langle u_{\sigma^{+}_{i_{1}}\cdots
\sigma^{+}_{i_{m-1}}i_{m}}\rangle= T_{i_{1}}\cdots
T_{i_{m-1}}\langle u_{i_{m}}\rangle = T''_{i_{1},1}\cdots
T''_{i_{m-1},1}(E_{i_{m}}).$$

(2)For any preinjective module $V_{\alpha}$, there exists a sink
sequence $i_{1},\cdots,i_{m}$ for $\Omega$ such that $$\langle
u_{\alpha}\rangle= \langle u_{\sigma^{-}_{i_{1}}\cdots
\sigma^{-}_{i_{m-1}}i_{m}}\rangle= T'_{i_{1}}\cdots
T'_{i_{m-1}}\langle u_{i_{m}}\rangle=T'_{i_{1},-1}\cdots
T'_{i_{m-1},-1}(E_{i_{m}}).$$
\end{prop}

\section{The elements in Hall algebras corresponding to exceptional
modules}\label{section 5} Let $\Lambda$ be a finite dimensional
hereditary $k$-algebra as in Section \ref{section 3}. Denote by
mod-$\Lambda$ the category of finite-dimensional $\Lambda$-modules.

Recall that a $\Lambda$-module $V_{\alpha}$ is called
\textit{exceptional} if
$\Ext^{1}_{\Lambda}(V_{\alpha},V_{\alpha})=0$. A pair of
indecomposable exceptional modules $(V_{\alpha},V_{\beta})$ is
called an \textit{exceptional pair} if
$\Hom_{\Lambda}(V_{\beta},V_{\alpha})=\Ext^{1}_{\Lambda}(V_{\beta},V_{\alpha})=0$.
A sequence of indecomposable $\Lambda$-modules
$(V_{\alpha_{1}},V_{\alpha_{2}},...,V_{\alpha_{n}})$ is called an
\textit{exceptional sequence} if any pair
$(V_{\alpha_{i}},V_{\alpha_{j}})$ with $i<j$ is exceptional. An
exceptional sequence
$(V_{\alpha_{1}},V_{\alpha_{2}},...,V_{\alpha_{n}})$ is said to be
\textit{complete} if $n=|I|$.

By Crawley-Boevey \cite{cb} and Ringel \cite{r2} we know that there
is a nice braid group action on the set of complete exceptional
sequences by which we can obtain all exceptional modules from an
exceptional sequence consisting of simple modules. We will briefly
recall the theory.

The braid group action is based on the following results: For any
exceptional sequence
$(V_{\alpha_{1}},V_{\alpha_{2}},...,V_{\alpha_{s}})$, let $\mathcal
{C}(\alpha_{1},\alpha_{2},...,\alpha_{s})$ be the smallest full
subcategory of mod-$\Lambda$ which contains
$V_{\alpha_{1}},V_{\alpha_{2}},...,V_{\alpha_{s}}$ and is closed
under extensions, kernels of epimorphisms and cokernels of
monomorphisms. $\mathcal{C}(\alpha_{1},\alpha_{2},...,\alpha_{s})$
is equivalent to mod-$\Lambda'$ where $\Lambda'$ is a finite
dimensional hereditary algebra with $s$ isomorphism classes of
simple modules. Furthermore, the canonical embedding of $\mathcal
{C}(\alpha_{1},\alpha_{2},...,\alpha_{s})$ into mod-$\Lambda$ is
exact and induces isomorphisms on both Hom and Ext.

In particular, the results above holds for any exceptional pair
$(V_{\alpha},V_{\beta})$. That is, $\mathcal{C}(\alpha,\beta)$ is
equivalent to the module category of a generalized Kronecker algebra
which has no regular exceptional modules. Hence
$(V_{\alpha},V_{\beta})$ have to be slice modules in the
preprojective component or preinjective component of
$\mathcal{C}(\alpha,\beta)$ or the orthogonal pair, i.e.
$V_{\alpha}$ is the simple injective and $V_{\beta}$ is the simple
projective. Thus, for an exceptional pair $(V_{\alpha},V_{\beta})$,
there are unique modules $L(\alpha,\beta)$ and $R(\alpha,\beta)$
such that $(L(\alpha,\beta),V_{\alpha})$ and
$(V_{\beta},R(\alpha,\beta))$ are exceptional pairs in
$\mathcal{C}(\alpha,\beta)$.

Let $\mathcal{V}=(V_{\alpha_{1}},V_{\alpha_{2}},...,V_{\alpha_{s}})$
be an exceptional sequence in mod-$\Lambda$. For $1\le i\le s$,
Define
$\sigma_{i}\mathcal{V}=(V_{\beta_{1}},V_{\beta_{2}},...,V_{\beta_{s}})$,
where
\begin{displaymath}
V_{\beta_{j}}=
\begin{cases}
V_{\alpha_{i+1}} &
\textrm{if $j=i$}\\
R(\alpha_{i},\alpha_{i+1}) & \textrm{if $j=i+1$}\\
V_{\alpha_{j}} & \textrm{if $j\notin\{i,i+1\}$}
\end{cases}
\end{displaymath}

Also, define
$\sigma^{-1}_{i}\mathcal{V}=(V_{\gamma_{1}},V_{\gamma_{2}},...,V_{\gamma_{s}})$,
where
\begin{displaymath}
V_{\gamma_{j}}=
\begin{cases}
L(\alpha_{i},\alpha_{i+1}) &
\textrm{if $j=i$}\\
V_{\alpha_{i}} & \textrm{if $j=i+1$}\\
V_{\alpha_{j}} & \textrm{if $j\notin\{i,i+1\}$}
\end{cases}
\end{displaymath}

Denote by $\mathcal{B}_{s-1}$ the group generated by
$\sigma_{1},\sigma_{2},...,\sigma_{s-1}$. The above definitions give
an action of $\mathcal{B}_{s-1}$ on the set of exceptional sequences
of length $s$. Moreover, $\sigma_{1},\sigma_{2},...,\sigma_{s-1}$
satisfy the braid relations:
\begin{displaymath}
\begin{cases}
\sigma_{i}\sigma_{i+1}\sigma_{i}=\sigma_{i+1}\sigma_{i}\sigma_{i+1}
& \text{for $1\le i\le s-1$}\\
\sigma_{i}\sigma_{j}=\sigma_{j}\sigma_{i} & \text{for $|i-j|\ge 2$}
\end{cases}
\end{displaymath}

So $\mathcal{B}_{s-1}$ is the braid group of $s-1$ generators. In
particular, in the case $s=|I|$, we have the braid group action on
the set of complete exceptional sequences. This action is transitive
according to Crawley-Boevey \cite{cb} and Ringel \cite{r2}.

Since any indecomposable exceptional module can be enlarged to a
complete exceptional sequence, we could obtain all indecomposable
exceptional modules via the braid group action from any given
complete exceptional sequence, in particular, the exceptional
sequence consisting of all the simple modules. In \cite{cx}, an
explicit inductive algorithm is given to express $\langle
u_{\lambda}\rangle$ as the combinations of elements $\langle
u_{i}\rangle$ if $V_{\lambda}$ is an indecomposable exceptional
module. We write down the formulas with some modifications, for the
definition of $_{\alpha}\delta$ and $\delta_{\alpha}$ in Section
\ref{section 3} is different from that in \cite{cx}.

\begin{thm}\label{thm 5 exc sequence}
For $1\le s\le |I|$, let
$\mathscr{B}=\langle\sigma_{1},\sigma_{2},...,\sigma_{s-1}\rangle$
be the braid group on $s-1$ generators,
$\mathcal{V}=(V_{\alpha_{1}},V_{\alpha_{2}},...,V_{\alpha_{s}})$ any
exceptional sequence of length $s$ in mod-$\Lambda$. Denote by
\begin{equation*}
m(i,i+1)=\frac{\langle\alpha_{i},\alpha_{i+1}\rangle}{\langle\alpha_{i+1},\alpha_{i+1}\rangle}=2\frac{(\alpha_{i},\alpha_{i+1})}{(\alpha_{i+1},\alpha_{i+1})}
\end{equation*}
and
\begin{displaymath}
n(i,i+1)=\frac{\langle\alpha_{i},\alpha_{i+1}\rangle}{\langle\alpha_{i},\alpha_{i}\rangle}=2\frac{(\alpha_{i},\alpha_{i+1})}{(\alpha_{i},\alpha_{i})}
\end{displaymath}
and assume that
$\sigma_{i}\mathcal{V}=(V_{\beta_{1}},V_{\beta_{2}},...,V_{\beta_{s}})$,
$\sigma^{-1}_{i}\mathcal{V}=(V_{\gamma_{1}},V_{\gamma_{2}},...,V_{\gamma_{s}})$
for $1\le i\le {s-1}$.

Then, in the Hall algebra $\mathscr{H}(\Lambda)$, we have

\emph{(1)} If $m(i,i+1)\dimv V_{\alpha_{i+1}}>\dimv V_{\alpha_{i}}$,
then
\begin{gather*}
\langle
u_{\beta_{i+1}}\rangle=\sum_{r=0}^{m(i,i+1)-1}(-1)^{r}v^{2\dim
V_{\alpha_{i}}}v^{\varepsilon(\alpha_{i})}(v^{-\varepsilon(\alpha_{i+1})})^{(m(i,
i+1)^{2}-m(i,i+1)r+r)}\\\times \langle
u_{\alpha_{i+1}}\rangle^{(r)}\delta_{\alpha_{i}}(\langle
u_{\alpha_{i+1}}\rangle^{(m(i,i+1)-r)}).
\end{gather*}

\emph{(2)} If $0<m(i,i+1)\dimv V_{\alpha_{i+1}}<\dimv
V_{\alpha_{i}}$, then
\begin{displaymath}
\langle u_{\beta_{i+1}}\rangle=\frac{v^{2m(i,i+1)\dim
V_{\alpha_{i+1}}}}{[m(i,i+1)]!_{\varepsilon(\alpha_{i+1})}}(_{\alpha_{i+1}}\delta)^{m(i,i+1)}(\langle
u_{\alpha_{i}}\rangle).
\end{displaymath}

\emph{(3)} If $m(i,i+1)\le 0$, then
\begin{displaymath}
\langle
u_{\beta_{i+1}}\rangle=\sum_{r=0}^{-m(i,i+1)}(-1)^{r}(v^{-r\varepsilon(\alpha_{i+1})})\langle
u_{\alpha_{i+1}}\rangle^{(r)}\langle u_{\alpha_{i}}\rangle\langle
u_{\alpha_{i+1}}\rangle^{(-m(i,i+1)-r)}.
\end{displaymath}

\emph{(1')} If $n(i,i+1)\dimv V_{\alpha_{i}}>\dimv
V_{\alpha_{i+1}}$, then
\begin{gather*}
\langle
u_{\gamma_{i}}\rangle=\sum_{r=0}^{n(i,i+1)-1}(-1)^{r}v^{2\dim
V_{\alpha_{i+1}}}v^{\varepsilon(\alpha_{i+1})}(v^{-\varepsilon(\alpha_{i})})
^{(n(i,i+1)^{2}-n(i,i+1)r+r)}\\\times(_{\alpha_{i+1}}\delta(\langle
u_{\alpha_{i}}\rangle^{(n(i,i+1)-r)}))\langle
u_{\alpha_{i}}\rangle^{(r)}.
\end{gather*}

\emph{(2')}If $0<n(i,i+1) \dimv V_{\alpha_{i}}<\dimv
V_{\alpha_{i+1}}$, then
\begin{displaymath}
\langle u_{\gamma_{i}}\rangle=\frac{v^{2n(i,i+1)\dim
V_{\alpha_{i}}}}{[n(i,i+1)]!_{\varepsilon(\alpha_{i})}}(\delta_{\alpha_{i}})^{n(i,i+1)}(\langle
u_{\alpha_{i+1}}\rangle).
\end{displaymath}

\emph{(3')}If $n(i,i+1)\le 0$, then
\begin{displaymath}
\langle
u_{\gamma_{i}}\rangle=\sum_{r=0}^{-n(i,i+1)}(-1)^{r}(v^{-r\varepsilon(\alpha_{i})})\langle
u_{\alpha_{i}}\rangle^{-n(i,i+1)-r}\langle
u_{\alpha_{i+1}}\rangle\langle u_{\alpha_{i}}\rangle^{(r)}.
\end{displaymath}
\end{thm}

By this theorem we can see that for any indecomposable exceptional
module $V_{\lambda}$, $\langle u_{\lambda}\rangle$ lies in the
generic composition algebra $\mathscr{C}(\Delta)$.

\section{Main results}\label{section 6}
We keep the notations as before. In this section we will identify
the quantum group $U$ (resp. the positive part $U^{+}$, the
$\mathbb{Z}$-form $U_{\mathbb{Z}}^{+}$) with the reduced Drinfeld
double $\mathcal{D}_{\mathscr{C}}(\Delta)$ (resp. the generic
composition algebra $\mathscr{C}(\Delta)$, the integral generic
composition algebra $\mathscr{C}_{\mathbb{Z}}(\Delta)$).

In Section \ref{section 5} we have seen that for any indecomposable
exceptional $\Lambda$-module $V_{\lambda}$, $\langle u_{\lambda}
\rangle$ lies in the quantum group. Our first result is a stronger
assertion which says that $\langle u_{\lambda}\rangle$ lies in the
integral form of the quantum group and here $V_{\lambda}$ can be any
(not only indecomposable) exceptional module.

\begin{thm}\label{thm 6 exc in integral comp alg}
Let $\Lambda$ be a finite-dimensional hereditary $k$-algebra,
$V_{\lambda}$ an exceptional $\Lambda$-module, then $\langle
u_{\lambda}\rangle$ lies in $U_{\mathbb{Z}}^{+}$.
\end{thm}

In Section \ref{section 8} we will prove that $\langle
u_{\lambda}\rangle\in L(\infty)$. So $\langle u_{\lambda}\rangle$
lies in $L_{\mathbb{Z}}(\infty)$. By abuse of language, we denote
the image of $\langle u_{\lambda}\rangle$ in
$L_{\mathbb{Z}}(\infty)/v^{-1}L_{\mathbb{Z}}(\infty)$ still by
$\langle u_{\lambda}\rangle$. Our second result is

\begin{thm}\label{thm 6 exc in crystal basis}
Let $\Lambda$ be a finite-dimensional hereditary $k$-algebra. Then
for any exceptional module $V_{\lambda}$,
\begin{displaymath}
\langle u_{\lambda}\rangle\in B(\infty)\cup(-B(\infty)).
\end{displaymath}
i.e. $\langle u_{\lambda}\rangle$ lies in the crystal basis up to a
sign.
\end{thm}

\section{Proof of Theorem 6.1}\label{section 7} Theorem \ref{thm 5 exc sequence} provides an
inductive method to express $\langle u_{\lambda}\rangle$ as
combinations of the Chevalley generators. If we can prove that in
each step the coefficients of the formulas are in
$\mathbb{Z}[v,v^{-1}]$, we are done immediately for proving Theorem
\ref{thm 6 exc in integral comp alg}. Unfortunately we could not
achieve this since the derivations $_{\alpha}\delta$ and
$\delta_{\alpha}$ are not $U_{\mathbb{Z}}$-stable. Instead, we will
use Lusztig's symmetries which is known to be
$U_{\mathbb{Z}}$-stable.

First we introduce some notations. For any exceptional pair
$(V_{\alpha},V_{\beta})$, the subcategory
$\mathcal{C}(\alpha,\beta)$ (Recall section \ref{section 5}) is
equivalent to mod-$\Lambda'$ for some finite dimensional hereditary
$k'$-algebra $\Lambda'$ which has two simple modules. Then the
corresponding Hall algebra $\mathscr{H}(\Lambda')$ is a subalgebra
of $\mathscr{H}(\Lambda)$ and the composition algebra
$\mathscr{C}(\Lambda')$ is a subalgebra of $\mathscr{C}(\Lambda)$
(note that the simple $\Lambda'$-modules viewed as $\Lambda$-modules
are exceptional). Denote the Cartan datum of $\Lambda'$ by
$\Delta'$. The generic composition algebra $\mathscr{C}(\Delta')$ is
also a subalgebra of $\mathscr{C}(\Delta)$. Now denote the quantum
group associated to $\Delta'$ by $U'$. Then we have an embedding
$U'^{+}\hookrightarrow U^{+}$. The discussion here means that an
exceptional pair gives a sub-Hall algebra of $\mathscr{H}(\Lambda)$
which corresponds to a sub-quantum group (positive part) of $U^{+}$.

Thus everything in Section \ref{section 3} and \ref{section 4} works
for $\Lambda'$. Suppose that $k'$ is a finite field with
$q'=(v')^{2}$ elements where $v'=v^{a}$ for some positive integers
$a$. For $V_{\gamma}\in\mathcal{C}(\alpha,\beta)$ we use the
notation $\langle
u_{\gamma}\rangle'=(v')^{-\dim_{k'}V_{\gamma}+\varepsilon'(\gamma)}u_{\gamma}$
in $\mathscr{H}(\Lambda')$ and the derivations
$\delta_{\gamma}',{_{\gamma}\delta'}$ are well-defined. Denote the
two simple $\Lambda'$-modules by $V_{s_{1}},V_{s_{2}}$, the
corresponding elements in $\mathscr{H}(\Lambda')$ by $\langle
u_{s_{1}}\rangle'$ and $\langle u_{s_{2}}\rangle'$. We have the
(relative) symmetries $T_{s_{1}}$, $T_{s_{2}}$ and $T_{s_{1}}'$,
$T_{s_{2}}'$ as in Section \ref{section 4}.

\begin{lem}\label{lem 7 relative derivation}
For any $V_{\gamma}\in\mathcal{C}(\alpha,\beta)$, we have
\begin{align*}
&(1)\ \ \ \langle u_{\gamma}\rangle'=\langle u_{\gamma}\rangle;\\
&(2)\ \ \
\delta_{\gamma}|_{\mathcal{C}(\alpha,\beta)}=\delta_{\gamma}',\ \ \
_{\gamma}\delta|_{\mathcal{C}(\alpha,\beta)}=\ _{\gamma}\delta',
\end{align*}
where $\delta_{\gamma}|_{\mathcal{C}(\alpha,\beta)}$ and
$_{\gamma}\delta|_{\mathcal{C}(\alpha,\beta)}$ denote the
restrictions of $\delta_{\gamma}$ and $_{\gamma}\delta$ to
$\mathscr{H}(\Lambda')$.
\end{lem}
\begin{proof}
(1) By definition
\begin{equation*}
\langle
u_{\gamma}\rangle=v^{-\dim_{k}V_{\gamma}+\varepsilon(\gamma)}u_{\gamma}.
\end{equation*}
The number
\begin{equation*}
v^{-\dim_{k}V_{\gamma}+\varepsilon(\gamma)}=\dfrac{\sqrt{|\End_{\Lambda}(V_{\gamma})|}}{\sqrt{|V_{\gamma}||\Ext_{\Lambda}(V_{\gamma},V_{\gamma})|}}
\end{equation*}
is unchanged whether we consider $V_{\gamma}$ as a $\Lambda$-module
or a $\Lambda'$-module since the embedding
$\mathcal{C}(\alpha,\beta)\hookrightarrow$ mod-$\Lambda$ induces
isomorphisms on both $\Hom$ and $\Ext$.

(2)\ We only prove the assertion for $\delta_{\gamma}$. Recall the
definition of $\delta_{\gamma}$ in 3.3.
\begin{equation*}
\delta_{\gamma}(u_{\lambda})=\sum_{\rho\in\mathcal{P}}v^{\langle
\gamma,\rho\rangle}g_{\gamma\rho}^{\lambda}\frac{a_{\rho}}{a_{\lambda}}u_{\rho}.
\end{equation*}
Since $\mathcal{C}(\alpha,\beta)$ is closed under kernels of
epimorphisms, we can see that if
$V_{\lambda}\in\mathcal{C}(\alpha,\beta)$, those $V_{\rho}$ with
$u_{\rho}$ occurring in the right hand side must be in
$\mathcal{C}(\alpha,\beta)$.

Note that
\begin{equation*}
v^{\langle
\gamma,\rho\rangle}=\dfrac{\sqrt{|\Hom_{\Lambda}(V_{\gamma},V_{\rho})|}}{\sqrt{|\Ext_{\Lambda}(V_{\gamma},V_{\rho})|}}.
\end{equation*}
Hence the numbers $v^{\langle \gamma,\rho\rangle}$,
$g_{\gamma\rho}^{\lambda}$, $a_{\rho}$ and $a_{\lambda}$ are
unchanged whether we consider $V_{\gamma},V_{\rho},V_{\lambda}$ as
$\Lambda$-modules or $\Lambda'$-modules.
\end{proof}

The following lemma is the key point of the proof.

\begin{lem}\label{lem 7 exc seq and symmetries}
For any indecomposable exceptional module $V_{\lambda}$, there exist
a positive integer $n$, subcategories $\mathscr{C}_{i}$ ($1\leq
i\leq n$) of mod-$\Lambda$ and indecomposable exceptional
$\Lambda$-modules $V_{\alpha_{i}}$ ($1\leq i \leq n+1$) where for
each $i$, $\mathscr{C}_{i}\simeq$ mod-$\Lambda_{i}$ (here
$\Lambda_{i}$ is a finite dimensional hereditary algebra having
exactly two simple modules $S_{i,1}$ and $S_{i,2}$) such that
$$\langle u_{\lambda}\rangle=\langle
u_{\alpha_{n+1}}\rangle;$$ and for each $1\leq i\leq n$,
$$\langle u_{\alpha_{i+1}}\rangle=T_{S_{i,i_{1}}}T_{S_{i,i_{2}}}\cdots T_{S_{i,i_{d_{i}}}}(\langle u_{\alpha_{i}}\rangle), \ or \ \langle u_{\alpha_{i+1}}\rangle=T_{S_{i,i_{1}}}'T_{S_{i,i_{2}}}'\cdots T_{S_{i,i_{d_{i}}}}'(\langle u_{\alpha_{i}}\rangle)$$
$$V_{\alpha_{i}}=S_{i,1}\ or \ S_{i,2}\ \ i.e. \ \langle u_{\alpha_{i}}\rangle=\langle u_{S_{i,1}}\rangle \ or \ \langle u_{S_{i,2}}\rangle$$
where $i_{1},i_{2},\cdots,i_{d_{i}}\in \{1,2\}$  for each $1\leq
i\leq
 n$.
\end{lem}
\begin{proof}
If $V_{\lambda}$ is simple, there is nothing to prove. So
assume $V_{\lambda}$ is not simple. Then by Ringel (See \cite{r5}
Section 8) there exists an exceptional sequence
$(V_{\lambda},V_{\mu})$ or $(V_{\mu},V_{\lambda})$ such that
$V_{\lambda}$ is not simple in $\mathcal{C}(V_{\lambda},V_{\mu})$.

We know that $\mathcal{C}(V_{\lambda},V_{\mu})$ is isomorphic to a
finite dimensional hereditary algebra $\Lambda'$ with just two
simple modules and $V_{\lambda}$ viewed as a $\Lambda'$-module is
preprojective or preinjective. Denote the simple $\Lambda'$-modules
by $V_{S_{1}}$ and $V_{S_{2}}$. By Proposition \ref{prop 4 obtain
irregular} we have
$$\langle u_{\lambda}\rangle=T_{S_{i_{1}}}T_{S_{i_{2}}}\cdots T_{S_{i_{m-1}}}(\langle u_{S_{i_{m}}}\rangle)$$
where $i_{1},i_{2},\cdots,i_{m} (i_{j}\in\{1,2\}, j=1,\cdots,m)$ is
a source sequence of the graph of $\Lambda'$, or
$$\langle u_{\lambda}\rangle=T_{S_{i_{1}}}'T_{S_{i_{2}}}'\cdots T_{S_{i_{m-1}}}'(\langle u_{S_{i_{m}}}\rangle)$$
where $i_{1},i_{2},\cdots,i_{m} (i_{j}\in\{1,2\}, j=1,\cdots,m)$ is
a sink sequence of the graph of $\Lambda'$.

Note that in the formulas above we should use $\langle \ \rangle'$,
but Lemma \ref{lem 7 relative derivation}(1) told us $\langle \
\rangle'=\langle \ \rangle$.

The lemma follows by induction.
\end{proof}

Now we can prove Theorem \ref{thm 6 exc in integral comp alg}.

First let us see that theorem \ref{thm 6 exc in integral comp alg}
holds for any indecomposable exceptional module $V_{\lambda}$. By
the above lemma, in each step of the induction we are in some
sub-quantum group, say $U'$. The symmetries are
$U_{\mathbb{Z}}'$-stable and the coefficient ring is
$\mathbb{Z}[v',v'^{-1}]$ where $v'=v^{a}$ for some positive integer
$a$. Obviously $\mathbb{Z}[v',v'^{-1}]\subset \mathbb{Z}[v,v^{-1}]$.
Hence $\langle u_{\lambda}\rangle$ lies in the integral composition
algebra. The only thing we may worry about is whether the
calculation is generic. But Theorem \ref{thm 5 exc sequence} has
ensured it. Thus $\langle u_{\lambda}\rangle$ is in
$U_{\mathbb{Z}}^{+}$.

Next we consider the case $V_{\lambda}\simeq sV_{\rho}$ where
$V_{\rho}$ is indecomposable exceptional (we also write $sV_{\rho}$
as $V_{s\rho}$). It is well known that (see \cite{r5}, for example)
we have
$$\langle u_{s\rho}\rangle=\langle u_{\rho}\rangle^{(s)},$$
Note that Lusztig's symmetries are endomorphisms hence the right
hand side is also in $U_{\mathbb{Z}}^{+}$ by Lemma \ref{lem 7 exc
seq and symmetries}.

Now we consider the general case. By induction we can reduce to the
case that $V_{\lambda}\simeq V_{s\mu}\oplus V_{t\nu}$ where
$V_{\mu}$ and $V_{\nu}$ are non-isomorphic indecomposable
exceptional modules.

We need two lemmas. One gives the calculation of the filtration
number $g_{\alpha\beta}^{\gamma}$, that is due to Riedtmann
\cite{rie} and Peng \cite{p}.

\begin{lem}\label{lem 7 Peng formula}
For any $\Lambda$-modules $V_{\alpha}$, $V_{\beta}$ and $V_{\gamma}$
$$g_{\alpha\beta}^{\gamma}=\frac{a_{\gamma}|\Ext_{\Lambda}(V_{\alpha},V_{\beta})_{V_{\gamma}}|}{a_{\alpha}a_{\beta}|\Hom_{\Lambda}(V_{\alpha},V_{\beta})|}$$
where $\Ext_{\Lambda}(V_{\alpha},V_{\beta})_{V_{\gamma}}$ is the set
of exact sequence in $\Ext_{\Lambda}(V_{\alpha},V_{\beta})$ with
middle term $V_{\gamma}$.
\end{lem}

The other one tells us how to compute the automorphism groups of
decomposable modules (See \cite{cx} or \cite{z}).

\begin{lem}\label{lem 7 calcul End of decomposables}
\emph{(1)}\ Let $V_{\lambda}$ be an indecomposable $\Lambda$-module
with $\dim_{k}\End_{\Lambda}(V_{\lambda})=s$ and $\dim_{k}\rad
\End_{\Lambda}(V_{\lambda})=t$, then
$$a_{\lambda}=(v^{2(s-t)}-1)v^{2t}$$

\emph{(2)}\ Let $V_{\lambda}\cong
s_{1}V_{\lambda_{1}}\oplus\cdots\oplus s_{t}V_{\lambda_{t}}$ such
that $V_{\lambda_{i}}\ncong V_{\lambda_{j}}$ for any $i\neq j$, then
$$a_{\lambda}=v^{2s}a_{s_{1}\lambda_{1}}\cdots a_{s_{t}\lambda_{t}},$$
where $s=\sum_{i\neq
j}s_{i}s_{j}\dim_{k}\Hom_{\Lambda}(V_{\lambda_{i}},V_{\lambda_{j}})$.

\emph{(3)}\ Let $V_{\lambda}=sV_{\rho}$ with
$\End_{\Lambda}(V_{\rho})=F$ and $F$ is an extension field of $k$,
then
$$a_{\lambda}=\prod_{0\leq t\leq s-1}(d^{s}-d^{t}),$$ where
$d=|F|=v^{2[F:k]}$.
\end{lem}

Now we can see that
\begin{equation*}
\begin{split}
\langle u_{s\mu}\rangle\langle u_{t\nu}\rangle & =v^{-\langle
t\nu,s\mu\rangle}g_{(s\mu)(t\nu)}^{s\mu\oplus t\nu}\langle
u_{s\mu\oplus t\nu}\rangle\\
& =v^{-\langle t\nu,s\mu\rangle}\frac{a_{s\mu\oplus
t\nu}}{a_{s\mu}a_{t\nu}|\Hom_{\Lambda}(V_{s\mu},V_{t\nu})|}\langle
u_{s\mu\oplus t\nu}\rangle\\
& =v^{-\langle
t\nu,s\mu\rangle}|\Hom_{\Lambda}(V_{t\nu},V_{s\mu})|\langle
u_{s\mu\oplus t\nu}\rangle
\end{split}
\end{equation*}

Hence we have
$$\langle
u_{\lambda}\rangle=\langle u_{s\mu\oplus t\nu}\rangle=v^{\langle
t\nu,s\mu\rangle-2st\dim_{k}\Hom_{\Lambda}(V_{\nu},V_{\mu})}\langle
u_{s\mu}\rangle\langle u_{t\nu}\rangle,$$ which lies in
$U_{\mathbb{Z}}^{+}$ and we are done.

\section{Proof of Theorem 6.2}\label{section 8}

\textbf{8.1 The pairing $(-,-)_{R}$.} \ \ Define a pairing
$(-,-)_{R}$: $\mathscr{H}(\Lambda)\times \mathscr{H}(\Lambda)
\rightarrow \mathbb{Q}(v)$ by
\begin{displaymath}
(\langle u_{\beta}\rangle,\langle
u_{\beta'}\rangle)_{R}=v^{(\beta,\beta)}a_{\beta}^{-1}\delta_{\beta\beta'},
\end{displaymath}
for all $\beta,\beta'\in \mathcal{P}$. (This pairing was first
proposed by Ringel, see \cite{r4})

Now we have two non-degenerate symmetric bilinear forms on $U^{+}$,
namely $(-,-)_{R}$ and $(-,-)_{K}$ (recall Section \ref{section 2}).
We will deduce some properties of $(-,-)_{R}$ and compare it with
$(-,-)_{K}$. Note that in Lemma \ref{lem 3 derivation coincide}, we
have proved that the derivations $r_{i}'$ and $f_{i}'$ coincide.

Denote $v^{\varepsilon(i)}$ by $v_{i}$. We have the following lemma.

\begin{lem}\label{lem 8 calcul Ringel form}
For any $i\in I$ and $\lambda_{1},\lambda_{2}\in \mathcal{P}$, we
have
\begin{displaymath}
(\langle u_{\lambda_{1}}\rangle,\langle u_{i}\rangle\langle
u_{\lambda_{2}}\rangle)_{R}=(1-v_{i}^{-2})^{-1}(r_{i}'(\langle
u_{\lambda_{1}}\rangle),\langle u_{\lambda_{2}}\rangle)_{R},
\end{displaymath}
\begin{displaymath}
(\langle u_{\lambda_{1}}\rangle,\langle
u_{\lambda_{2}}\rangle\langle
u_{i}\rangle)_{R}=(1-v_{i}^{-2})^{-1}(\langle
u_{\lambda_{1}}\rangle,r_{i}(\langle u_{\lambda_{2}}\rangle))_{R}.
\end{displaymath}
\end{lem}
\begin{proof}
We only prove the first one. The proof of the second one is
similar. By definition, we know that
$$\langle u_{i}\rangle\langle u_{\lambda_{2}}\rangle=v^{-\langle\lambda_{2},i\rangle}\sum_{\lambda\in\mathcal{P}}g_{i\lambda_{2}}^{\lambda}\langle u_{\lambda}\rangle.$$

So we have
$$(\langle u_{\lambda_{1}}\rangle,\langle u_{i}\rangle\langle
u_{\lambda_{2}}\rangle)_{R}=v^{(\lambda_{1},\lambda_{1})-\langle\lambda_{2},i\rangle}\frac{g_{i\lambda_{2}}^{\lambda_{1}}}{a_{\lambda_{1}}}.$$

On the other hand, we have
$$r'_{i}(\langle u_{\lambda_{1}} \rangle)= \sum_{\beta\in \mathcal{P}}v^{\langle i,\beta
 \rangle+(i,\beta)}g^{\lambda_{1}}_{i\beta}\frac{a_{\beta}a_{i}}{a_{\lambda_{1}}}\langle u_{\beta}\rangle,$$

Thus we have
$$(r_{i}'(\langle
u_{\lambda_{1}}\rangle),\langle
u_{\lambda_{2}}\rangle)_{R}=v^{\langle i,\lambda_{2}
 \rangle+(i,\lambda_{2})+(\lambda_{2},\lambda_{2})}\frac{g^{\lambda_{1}}_{i\lambda_{2}}a_{i}}{a_{\lambda_{1}}}.$$

Note that
$$(\lambda_{1},\lambda_{1})=(\lambda_{2}+i,\lambda_{2}+i)=(\lambda_{2},\lambda_{2})+2(\lambda_{2},i)+(i,i)$$
and
$$a_{i}=(v_{i}^{2}-1)=v^{(i,i)}(1-v_{i}^{-2})^{-1},$$
we are done.
\end{proof}

By Lemma \ref{lem 8 calcul Ringel form} we can compute the pairing
$(-,-)_{R}$ inductively, similar to Proposition \ref{prop 2
Kashiwara form}.

\begin{lem}\label{lem 8 Ringel form}
For any $x,y\in U^{+}$, we have
\begin{gather*}
(1,1)_{R}=1,\\
(E_{i}x,y)_{R}=(1-v_{i}^{-2})^{-1}(x,f_{i}'(y))_{R}.
\end{gather*}
\end{lem}
\begin{proof}
By definition and Lemma \ref{lem 8 calcul Ringel form} we have
\begin{equation*}
\begin{split}
(E_{i}x,y)_{R} &=(\langle
u_{i}\rangle x,y)_{R}=(1-v_{i}^{-2})^{-1}(x,r_{i}'(y))_{R}\\
&=(1-v_{i}^{-2})^{-1}(x,f_{i}'(y))_{R}.\\
\end{split}
\end{equation*}
\end{proof}

\begin{lem}\label{lem 8 form of L in A}
For any $x,y\in U^{+}$, we have $(x,y)_{K}\in A$ if and only if
$(x,y)_{R}\in A$. In particular,
\begin{equation*}
(L(\infty),L(\infty))_{R}\subset A.
\end{equation*}
\end{lem}
\begin{proof}
First we have
\begin{equation*}
(1,1)_{R}=(1,1)_{K}=1.
\end{equation*}
For any $x,y\in U^{+}$, by Lemma \ref{lem 8 Ringel form}, we know
\begin{equation*}
(E_{i}x,y)_{R}=(1-v_{i}^{-2})^{-1}(x,f_{i}'(y))_{R}.
\end{equation*}
Since $(1-v_{i}^{-2})^{-1}$ and $1-v_{i}^{-2}$ are both in $A$, we
have $(E_{i}x,y)_{R}\in A$ if and only if $(x,f_{i}'(y))_{R}\in A$.
On the other hand. by Proposition \ref{prop 2 Kashiwara form},
\begin{equation*}
(E_{i}x,y)_{K}=(x,f_{i}'(y))_{K}.
\end{equation*}
Now the assertion of this lemma follows immediately by induction.
\end{proof}

Hence the form $(-,-)_{R}$ induces a $\mathbb{Q}$-bilinear form
$(-,-)_{R,0}$ on $L(\infty)/v^{-1}L(\infty)$:
\begin{displaymath}
(x+v^{-1}L(\infty),y+v^{-1}L(\infty))_{R,0}=(x,y)_{R}+v^{-1}A
\end{displaymath}
for any $x,y\in L(\infty)$. The next proposition says that the two
pairings $(x,y)_{R,0}$ and $(x,y)_{K,0}$ coincide.

\begin{prop}\label{prop 8 forms coincide}
For any $x,y\in L(\infty)$, if we denote their images in
$L(\infty)/v^{-1}L(\infty)$ still by $x,y$, then we have
\begin{equation*}
(x,y)_{R,0}=(x,y)_{K,0}.
\end{equation*}
\end{prop}
\begin{proof}
Just compare Proposition \ref{prop 2 Kashiwara form} and Lemma
\ref{lem 8 form of L in A}, and note that
\begin{equation*}
\dfrac{1}{1-v_{i}^{-2}}=1+\dfrac{v_{i}^{-2}}{1-v_{i}^{-2}}\in \
1+v^{-1}A.
\end{equation*}
\end{proof}

Thus we can use the pairing $(-,-)_{R}$ instead of $(-,-)_{K}$ to
characterize the crystal bases.

\begin{lem}\label{lem 8 char of L and B by Ringel form}
We have the following characterizations of $L(\infty)$ and
$B(\infty)$:
\begin{gather*}
L(\infty)=\{x\in U^{+}|(x,x)_{R}\in A\},\\
B(\infty)\cup(-B(\infty))=\{x\in
L_{\mathbb{Z}}(\infty)/v^{-1}L_{\mathbb{Z}}(\infty)|(x,x)_{R,0}=1\}.
\end{gather*}
\end{lem}
\begin{proof}
Recall Proposition \ref{prop 2 char of L}(b) and \ref{prop 2 char of
B}(c). The assertion of this corollary is an easy consequence of
Lemma \ref{lem 8 form of L in A} and Proposition \ref{prop 8 forms
coincide}.
\end{proof}

\textbf{8.2 The proof.}\ \ We need to calculate the pairing
$(\langle u_{\lambda}\rangle, \langle u_{\lambda}\rangle)_{R}$ for
each exceptional module $V_{\lambda}$.

First we assume that $V_{\lambda}$ is indecomposable. In this case
the calculation is easy since the endomorphism ring of $V_{\lambda}$
is a division ring. namely
\begin{equation*}
\begin{split}
(\langle u_{\lambda}\rangle, \langle u_{\lambda}\rangle)_{R} &
=\frac{v^{(\lambda,\lambda)}}{a_{\lambda}}=\frac{|\End_{\Lambda}(V_{\lambda})|}{|\Aut_{\Lambda}(V_{\lambda})|}\\
&
=\frac{v^{2\varepsilon(\lambda)}}{v^{2\varepsilon(\lambda)}-1}=\frac{1}{1-v^{-2\varepsilon(\lambda)}}\\
&
=1+\frac{v^{-2\varepsilon(\lambda)}}{1-v^{-2\varepsilon(\lambda)}}\
\in 1+v^{-1}A.
\end{split}
\end{equation*}

Secondly, for n-copies of an indecomposable $V_{\lambda}$ we get
(using Lemma \ref{lem 7 calcul End of decomposables}(3)):
\begin{equation*}
\begin{split}
(\langle u_{\lambda}\rangle^{(n)}, \langle
u_{\lambda}\rangle^{(n)})_{R} & =(\langle
u_{n\lambda}\rangle,\langle
u_{n\lambda}\rangle)=\frac{|\End_{\Lambda}(V_{n\lambda})|}{a_{n\lambda}}\\
& =\frac{v^{2n^{2}\varepsilon(\lambda)}}{\prod_{0\leq t\leq
n-1}(v^{2n\varepsilon(\lambda)}-v^{2t\varepsilon(\lambda)})}\\
& =\prod_{0\leq t\leq n-1}
\frac{1}{1-{v^{-2(n-t)\varepsilon(\lambda)}}}\ \in 1+v^{-1}A.
\end{split}
\end{equation*}

Now we can deal with any exceptional module $V_{\lambda}$. Assume
$V_{\lambda}\cong s_{1}V_{\lambda_{1}}\oplus\cdots\oplus
s_{t}V_{\lambda_{t}}$ such that $V_{\lambda_{i}}$ is indecomposable
for any $i$ and $V_{\lambda_{i}}\ncong V_{\lambda_{j}}$ for any
$i\neq j$. Thus
$$|\End_{\Lambda}(V_{\lambda})|=v^{2\sum_{i=1}^{t}s_{i}^{2}\varepsilon(\lambda_{i})+2s}$$
where $s=\sum_{i\neq
j}s_{i}s_{j}\dim_{k}\Hom_{\Lambda}(V_{\lambda_{i}},V_{\lambda_{j}})$.

By lemma \ref{lem 7 calcul End of decomposables} (2) we have
$a_{\lambda}=v^{2s}a_{s_{1}\lambda_{1}}\cdots a_{s_{t}\lambda_{t}}$,
hence

\begin{equation*}
\begin{split}
(\langle u_{\lambda}\rangle,\langle u_{\lambda}\rangle)_{R} &
=\frac{v^{2\sum_{i=1}^{t}s_{i}^{2}\varepsilon(\lambda_{i})+2s}}{v^{2s}a_{s_{1}\lambda_{1}}\cdots
a_{s_{t}\lambda_{t}}}=\frac{v^{2\sum_{i=1}^{t}s_{i}^{2}\varepsilon(\lambda_{i})}}{a_{s_{1}\lambda_{1}}\cdots
a_{s_{t}\lambda_{t}}}\\
&
=\frac{v^{2\sum_{i=1}^{t}s_{i}^{2}\varepsilon(\lambda_{i})}}{\prod_{i=1}^{t}\prod_{0\leq
t_{i}\leq
s_{i}-1}(v^{2s_{i}\varepsilon(\lambda_{i})}-v^{2t_{i}\varepsilon(\lambda_{i})})}\\
& =\prod_{i=1}^{t}\prod_{0\leq t_{i}\leq
s_{i}-1}(\frac{1}{1-v^{-2(s_{i}-t_{i})\varepsilon(\lambda_{i})}}) \
\in 1+v^{-1}A.
\end{split}
\end{equation*}

Thus we have proved that for any exceptional module $V_{\lambda}$,
$$(\langle u_{\lambda}\rangle,\langle u_{\lambda}\rangle)_{R}\in
A\ \ \text{and}\ \ (\langle u_{\lambda}\rangle,\langle
u_{\lambda}\rangle)_{R,0}=1.$$

Hence by Theorem \ref{thm 6 exc in integral comp alg} and Lemma
\ref{lem 8 char of L and B by Ringel form} we reach our goal.

\begin{rem}\label{rem 8 rank 2 case}
(1). In the rank $2$ case (i.e. $|I|=2$ in the Cartan datum ), the
sign can be removed (See \cite{s} and \cite{l2}). However, it seems
that their methods do not work in general.

(2). The formulas given in Theorem \ref{thm 5 exc sequence} can be
viewed as an inductive algorithm to obtain certain crystal basis
elements from the Chevalley generators.
\end{rem}

\section{Remove the sign}\label{section 9}
We have shown by algebraic methods that the elements corresponding
to exceptional modules lie in the crystal bases up to a sign.
Intuitively the sign should be removed since it does for rank $2$
case. In this section we will remove the sign with geometric methods
due to Lusztig. For convenience, we only consider the case of
symmetric Cartan datum.

So let us consider a quiver without cycles, that is $Q=(I,H,s,t)$,
where $I$ is the vertex set, $H$ is the arrow set and two maps $s,t$
indicate the start points and terminal points of arrows
respectively. let $F_{q}$ denotes a finite field of $q=p^{e}$
elements, where $p$ is a prime number. As in $3.1$, we get a Hall
algebra from the representation of $\Lambda = F_{q}Q$, hence a
realization of $U^{+}$.

Now we give a short review of two constructions by Lusztig. The
first is the geometrical construction of Hall algebras. For any
finite dimensional $I$-graded $F_{q}$-vector space $W=\sum_{i\in
I}W_{i}$, consider the moduli space of representation of $Q$:
$$E_{W}=\bigoplus_{\rho \in H} \Hom(W_{s(\rho)},W_{t(\rho)})$$ The
group $G_{W}=\prod_{i \in I}GL(W_{i})$ acts on $E_{W}$ naturally.
Let $\mathbb{C}_{G}(E_{W})$ be the space of $G_{W}$-invariant
functions $E_{W}\rightarrow \mathbb{C}$. For $\underline{c} \in
\mathbb{N}I$, we fix a $I$-graded $F_{q}$-vector space
$W_{\underline{c}}$ with $\dimv W_{\underline{c}}=\underline{c}$,
and denote $E_{\underline{c}}=E_{W_{\underline{c}}},
G_{\underline{c}}=G_{W_{\underline{c}}}$. Then the multiplication
can be defined in the $\mathbb{C}$-space $K=\bigoplus_{\underline{a}
\in \mathbb{N}I}\mathbb{C}_{G}(E_{\underline{a}})$, which makes $K$
an associative $\mathbb{C}$-algebra. Corresponding to $\alpha \in
\mathcal{P}$, let $\mathcal{O}_{\alpha}\subset E_{\underline{a}}$ be
the $G_{\underline{a}}$-orbit of module $V_{\alpha}\in
E_{\underline{a}}$. We take $1_{\alpha}\in
\mathbb{C}_{G}(E_{\underline{a}})$ to be the characteristic function
of $\mathcal{O}_{\alpha}$, and set $f_{\alpha}= v_{q}^{-\dim
\mathcal{O}_{\alpha}}1_{\alpha}$ where $v_{q}=\sqrt{q}$. Thus $K$ is
just the so-called Hall algebra if $f_{\alpha}$ is identified with
$\langle u_{\alpha}\rangle$ for all $\alpha \in \mathcal{P}$.

The second is the construction of $U^{+}$ in terms of perverse
sheaves. $E_{W}$ can be defined over an algebraic closure of the
finite field  $F_{p}$ of $p$ elements. $E_{W}$ has  a natural
$F_{p^{e}}$-structure with Frobenius map: $F^{e}: E_{W}\rightarrow
E_{W}$. Thus the $F_{q}$-rational points $E^{F^{e}}_{W}$ of $E_{W}$
provide an $F_{q}$-structure as above. Let
$D(E_{W})=D^{b}_{c}(E_{W})$ be the bounded derived category of
$\overline{\mathbb{Q}_{l}}$- constructible sheaves; here $l$ is a
fixed prime number distinct from $p$ and $\overline{\mathbb{Q}_{l}}$
is an algebraic closure of the field of $l$-adic numbers. Let
$P_{W}$ be the set of isomorphism classes of simple perverse sheaves
subject to some extra condition (see \cite{l1}). Then we can
associate $P_{W}$ a full subcategory $Q_{W}$ of $D(E_{W})$ and let
$K_{W}$ be the Grothendieck group of $Q_{W}$. Hence $\coprod_{W}
K_{W}$ gives a realization  of $U^{+}$ with $\coprod_{W} P_{W}$ as
its bases, that is the canonical bases.

For an exceptional module $V_{\lambda}$, $\mathcal{O}_{\lambda}$ is
the corresponding orbit. We know from \cite{l3} that the
intersection cohomology complex
$$IC(\mathcal{O_{\lambda}},\overline{\mathbb{Q}_{l}})=j_{!*}(\overline{\mathbb{Q}_{l}})[\dim
\mathcal{O}_{\lambda}]$$ belongs to the canonical bases $P_{W}$ of
$K_{W}$ where $j$ is the natural embedding
$\mathcal{O}_{\lambda}\rightarrow E_{W}$.

In \cite{l4}, Lusztig considered the correspondence between
functions on the moduli space and the elements of the canonical
bases. In other words, he showed what kind of functions lies in the
canonical bases when comparing the two constructions mentioned
above. We will use the notations in \cite{l4} from now on.

We only need to consider a special kind of canonical bases, i.e.
$b=IC(\mathcal{O_{\lambda}},\overline{\mathbb{Q}_{l}})$. $b$
corresponds to $(b_{e})_{e\geq1}$, $b_{e}: E^{F^{e}}_{W}\rightarrow
\overline{\mathbb{Q}_{l}}$. Sometimes $\overline{\mathbb{Q}_{l}}$
can be identified with $\mathbb{C}$. For $x \in E^{F^{e}}_{W}$,
$b_{e}(x)$ is the alternative sum of the trace of the induced
Frobenius  map on the stalk  at $x$ of the $i$-th cohomology sheaf
of $IC(\mathcal{O_{\lambda}},\overline{\mathbb{Q}_{l}})$.
$$b_{e}(x)=\sum_{i\in \mathbb{Z}}(-1)^{i}Tr(F^{e};H_{x}^{i}(j_{!*}(\overline{\mathbb{Q}_{l}})[\dim \mathcal{O_{\lambda}}])$$

Now that $\mathcal{O}_{\lambda}$ is an open dense subset of $E_{W}$,
we can deduce as follows.

Firstly, the result below is due to Gabber:
$$Tr(F^{e};H_{x}^{i}(j_{!*}(\overline{\mathbb{Q}_{l}})[\dim \mathcal{O_{\lambda}}]) \in q^{-\frac{\dim \mathcal{O}_{\lambda}}{2}}\mathbb{Z}[q^{-1}]$$

For $x \in E_{W}$, $F^{e}(x)=x \in \mathcal{O}_{\lambda}$, we have
$$Tr(F^{e};H_{x}^{i}(j_{!*}(\overline{\mathbb{Q}_{l}})[\dim \mathcal{O_{\lambda}}])= Tr(F^{e};H_{x}^{i}(\mathcal{O}_{\lambda},\overline{\mathbb{Q}_{l}}[\dim \mathcal{O}_{\lambda}])=q^{-\frac{\dim \mathcal{O}_{\lambda}}{2}}\delta_{i,dim \mathcal{O}_{\lambda}}$$

While for $y \in E_{W}$, $F^{e}(y)=y \notin \mathcal{O}_{\lambda}$,
any open neighborhood $U_{y}$ of $y$,
$$U_{y}\cap \mathcal{O}_{\lambda}= U'_{y}\neq \varnothing  \Longrightarrow H_{y}^{i}(j_{!*}(\overline{\mathbb{Q}_{l}})[\dim \mathcal{O_{\lambda}}])\cong H_{x}^{i}(j_{!*}(\overline{\mathbb{Q}_{l}})[\dim \mathcal{O_{\lambda}}])$$

So $b_{e}$ is actually a constant function
$$b_{e}= q^{-\frac{\dim \mathcal{O}_{\lambda}}{2}}(-1)^{\dim \mathcal{O}_{\lambda}}= (-q^{\frac{1}{2}})^{-\dim \mathcal{O}_{\lambda}}$$

We remark that $v=-q^{\frac{1}{2}}$ in the literature of \cite{l4},
so $b_{e}=v^{-\dim \mathcal{O}_{\lambda}}$. let $\mathcal{O}_{y}$ be
the orbit of $y$, we obtain:
$$b_{e}|_{\mathcal{O}_{\lambda}} = v^{-\dim \mathcal{O}_{\lambda}}1_{\mathcal{O}_{\lambda}}$$
$$b_{e}|_{\mathcal{O}_{y}} = v^{-\dim E_{W}}1_{\mathcal{O}_{y}} = v^{-(\dim E_{W}-\dim \mathcal{O}_{y})} v^{-\dim \mathcal{O}_{y}}1_{\mathcal{O}_{y}}$$

This means the image of $b_{e}$ in
$L_{\mathbb{Z}}(\infty)/v^{-1}L_{\mathbb{Z}}(\infty)$ is equal to
that of $\langle u_{\lambda}\rangle$ in
$L_{\mathbb{Z}}(\infty)/v^{-1}L_{\mathbb{Z}}(\infty).$ Hence we
obtain the finally result, which is a stronger version of Theorem
6.2.

\begin{thm}\label{thm 9 remove the sign}
Let $\Lambda$ be a finite-dimensional hereditary $k$-algebra. Then
for any exceptional module $V_{\lambda}$, we have
\begin{displaymath}
\langle u_{\lambda}\rangle\in B(\infty).
\end{displaymath}
i.e. the image of $\langle u_{\lambda}\rangle$ in
$L_{\mathbb{Z}}(\infty)/v^{-1}L_{\mathbb{Z}}(\infty)$ lies in the
crystal basis.
\end{thm}

\section*{Acknowledgement} We thank Guanglian Zhang for his
help in completing the last section.

\bigskip
\noindent Yong Jiang, Department of Mathematical Sciences, Tsinghua
University, Beijing 10084, P.~R.~China\\ Email address:
jiangy00@mails.tsinghua.edu.cn
\bigskip

\noindent Jie Sheng, Department of Mathematical Sciences, Tsinghua
University, Beijing 10084, P.~R.~China\\ Email address:
shengjie00@mails.tsinghua.edu.cn
\bigskip

\noindent Jie Xiao, Department of Mathematical Sciences, Tsinghua
University, Beijing 10084, P.~R.~China\\ Email address:
jxiao@math.tsinghua.edu.cn

\end{document}